\documentclass[final]{elsarticle}

\journal{Computers and Mathematics with Applications}

\usepackage{color}
\usepackage{subfigure}
\usepackage{amsmath,wasysym,amssymb}
\usepackage{amsthm}
\usepackage{longtable}
\usepackage{tabularx}
\usepackage{float}
\usepackage{tikz}
\usepackage{tikz-3dplot}
\usetikzlibrary{hobby}
\usetikzlibrary{positioning}
\usetikzlibrary{shapes}
\usetikzlibrary{plotmarks}
\usetikzlibrary{shapes.geometric, calc}

\usepackage[margin=2.5cm]{geometry} 

\newcommand{\vq}{ {\vec q } }
\newcommand{\vv}{ {\vec v } }
\newcommand{\vn}{ {\vec n } }

\newtheorem{ex}{Example}[subsection]

\newtheorem{theorem}{Theorem}[section]

\newtheorem{remark}{Remark}

\newtheorem{lemma}{Lemma}[section]

\begin{document}

\begin{frontmatter}

\title{Weak Galerkin finite element method for second order problems on curvilinear polytopal meshes with Lipschitz continuous edges or faces  }
% \tnotetext[mytitlenote]{The second author is supported by Grant R01EB034143. The third author is supported by the National Natural Science Foundation of China, No. 12001325.}

\author{Qingguang Guan}
\address{School of Mathematics and Natural Sciences,
	University of Southern Mississippi,
	Hattiesburg, MS 39406, USA}
%\fntext[myfootnote]{This work is supported by the National Natural Science Foundation of China Grant No. 12001325.}
\ead{qingguang.guan@usm.edu}
\author{Gillian Queisser}
\address{Department of Mathematics, Temple University, Philadelphia, PA 19122, USA}
\ead{gillian.queisser@temple.edu}

\author{Wenju Zhao\corref{mycorrespondingauthor}}
\address{School of Mathematics,
	Shandong University,
	Jinan, Shandong 250100, China}
%\fntext[myfootnote]{This work is supported by the National Natural Science Foundation of China Grant No. 12001325.}
\ead{zhaowj@sdu.edu.cn}
\cortext[mycorrespondingauthor]{Corresponding author}

\begin{abstract}
		In this paper, we propose new basis functions defined on curved sides or faces of curvilinear elements (polygons or polyhedrons with curved sides or faces) for the weak Galerkin finite element method. Those basis functions are constructed by collecting linearly independent traces of polynomials on the curved sides/faces. We then analyze the modified weak Galerkin method for the elliptic equation and the interface problem on curvilinear polytopal meshes with Lipschitz continuous edges or faces. The method is designed to deal with less smooth complex boundaries or interfaces. Optimal convergence rates for $H^1$ and $L^2$ errors are obtained, and arbitrary high orders can be achieved for sufficiently smooth solutions. The numerical algorithm is discussed and tests are provided to verify theoretical findings. 
\end{abstract}

\begin{keyword}
		Lipschitz Continuous Boundaries or Interfaces; Curvilinear Elements; Weak Galerkin Method; Traces of Polynomials; High Orders;  Second-Order PDEs
%\texttt{elsarticle.cls}\sep \LaTeX\sep Elsevier \sep template
\MSC[2010] 
65N30, 65N12, 35J25
\end{keyword}
\end{frontmatter}

 % \linenumbers

\section{Introduction}\label{introd}
It has been a challenging question: how to design high-order numerical methods solving {  Partial Differential Equations (PDEs)} on domains (2D or 3D) with less smooth curved boundaries or interfaces if the solutions are sufficiently smooth? We refer \cite{guyomarc2009discontinuous,guo2020immersed,guo2021solving,cockburn2014priori,he2023error} for high-order methods with piecewise $C^2$ boundaries or interfaces, and \cite{09iso,da2019virtual} for high-order methods relying on the smoothness of the boundaries/interfaces. Many numerical methods fall within those two categories. Our approach based on the weak Galerkin finite element method only needs boundaries or interfaces to be Lipschitz continuous, and the convergence rates depend on the solution, not the geometry. The weak Galerkin method \cite{wang2013weak,Wang14,Lin15,Lin16} among others (e.g., HDG \cite{cockburn2010projection,cockburn2014priori}, nonconforming VEM \cite{de2016nonconforming,dedner2022robust}, HHO\cite{solano2019high,burman2021unfitted}) as a new variation of Galerkin methods has been applied to various partial differential equations \cite{Wang14,mu2021wgcv,guan2018w,guan2020w}. The main advantages of those methods lie in using discontinuous basis functions inside the element and on its boundaries, employing general shape elements, ensuring mass conservation, and producing continuous numerical fluxes.  HDG, HHO, VEM, and other methods with curved elements dealing with elliptic interface problems and Dirichlet boundary value problems can be found in \cite{cockburn2010projection,cockburn2014priori,solano2019high,burman2021unfitted,da2019virtual,bertoluzza2022weakly} and references therein. The method proposed here is significantly different from them, it uses fitted meshes, and there is no need to map to reference elements and redefine basis functions. 
An early version of this paper in 2019, without numerical experiments, can be found on {\bf arXiv}, see \cite{guan2019w}. {  Similar idea was also introduced in the Virtual Element context for 2D problems, see \cite{beirao2020polynomial}}. Later development by other researchers includes \cite{mu2021wgcv,li2022curved,yang2022weak,yemm2022new}, but our method is still unique with relaxed requirements for regularities of boundaries/interfaces {  and uniform treatment in 2D and 3D}.
% We define the $L^2$ norm as $\|\cdot\|,$ the inner product as $(\cdot,\cdot)$, and the vector-valued space $H({\rm div};\Omega)$ as
% $$
% H({\rm div};\Omega) = \left\{ \vv: \vv\in [L^2(\Omega)]^n, \nabla\cdot\vv\in L^2(\Omega)\right\}.
% $$ 
The key idea of the weak Galerkin finite element method is the definition of ``weak gradient''. Suppose we have a curvilinear polygonal or polyhedral domain $D \subset \mathbb{R}^n, (n=2,3)$ with interior part $D_0$ and $C^{0,1}$ boundary $\partial D$, see Figure \ref{fg1} for a 2D example. A discontinuous function $v=(v_0,v_b)$ on $D$ is defined as: $v_0\in L^2(D_0)$, $v_b\in L^2(\partial D)$.  Then we denote $W(D)$ as the space of those discontinuous functions
\begin{equation}\label{wd}
	W(D) = \left\{v = (v_0,v_b): v_0\in L^2(D_0), v_b\in L^2(\partial D)\right\}.
\end{equation}
For any $v\in W(D)$, the ``weak gradient" of $v$ is a linear functional
in the dual space of $[H^1(D)]^n$, it's defined by $\nabla_{w}v$ in \eqref{wg1}
\begin{equation}\label{wg1}
	(\nabla_{w}v,\vq\ )_D := -\int_D v_0\nabla\cdot\vq\ {\rm d}x
	+
	\int_{\partial D} v_b \vq\cdot\vn\ {\rm d}S,
	\quad \forall \vq\in [H^1(D)]^n,
\end{equation}
where $\vn$ is the outward normal vector to $\partial D.$ 
The discrete ``weak gradient" $\nabla_{w,k,D} v$ of $v\in W(D)$ is defined as the solution of equation \eqref{wg2}
\begin{equation}\label{wg2}
	(\nabla_{w,k,D}v,\vq\ )_D = -\int_D v_0\nabla\cdot\vq\ {\rm d}x
	+
	\int_{\partial D} v_b \vq\cdot\vn\ {\rm d}S,
	\quad \forall \vq\in [\mathbb{P}_k(D)]^n,
\end{equation} 
where  $\nabla_{w,k,D} v \in [\mathbb{P}_k(D)]^n$, $\mathbb{P}_k(D)$ is the polynomial space with degree no more than $k$, {  where} $k$ is a non-negtive integer.  For brevity, in following sections, we use $\nabla_{w}$ to represent $\nabla_{w,k,D}$.
%%%%%%%%%%%%%%%%%%%%%%%%%%%%%%%%%%%%%%%%%%%%%%%%%%%%%%%%%%%%%%%%%%%%%%%%%%%%%%%%%%%%%%%%%%%%%%%%%%%%%
\begin{figure}[H]
	\begin{center}
		\begin{tikzpicture}[scale = 1.2]
			\coordinate (A) at (0,0);
			\coordinate (B) at (3,0);
			\coordinate (C) at (3.5,1.5);
			\coordinate (D) at (2,3);
			\coordinate (E) at (0,1);
			\draw plot [smooth, tension=0.8] coordinates { (D) (1.1,1.7) (E)};
			\coordinate (O) at ($1/5*(A)+1/5*(B)+1/5*(C)+1/5*(D)+1/5*(E)$);
			\draw (A)
			--(B)
			--(C)
			--(D);
			\draw (E)--(A);
			\draw[style=dashed](O) circle (0.7);
			\fill [black] (O) circle (1pt);
			\draw[style=dashed](O)--(A);
			\draw[style=dashed](O)--(B);
			\draw[style=dashed](O)--(C);
			\draw[style=dashed](O)--(D);
			\draw[style=dashed](O)--(E);
		\end{tikzpicture}
	\end{center}\caption{A star-shaped curvilinear element $D$ with a $C^{0,1}$ curved side}\label{fg1}
\end{figure}
The paper is structured as follows. In Section \ref{shape-regular}, the shape-regular assumptions for curvilinear elements are given, and new $L^2$ projections on curved sides/faces are defined. Based on the assumptions and projections,  necessary lemmas are proved. Then the weak Galerkin finite element scheme and new basis functions are proposed. Section \ref{poisson} is devoted to {  solving} Poisson's equation. The newly defined basis functions are used to solve the equation on curvilinear polytopal meshes with Lipschitz continuous edges or faces, {  where} optimal convergence rates in $L^2$ and $H^1$ norms are proved. 
%In Section \ref{interface}, the proposed weak Galerkin method is applied to the elliptic interface problem with Lipschitz continuous curved interfaces/boundaries, error analysis is given with optimal convergence rates.  
{  In Section \ref{interface}, we apply the proposed weak Galerkin method to solve the elliptic interface problem, which involves Lipschitz continuous curved interfaces/boundaries. We also provide an error analysis that demonstrates optimal convergence rates.}
Section \ref{tests} shows the numerical results in 2D. Problems in 3D can be solved similarly but not presented here. Poisson's equation with the Neumann boundary condition on the curved domain is also considered. The elliptic interface problem with the non-homogeneous jump condition on the interface is tested. For convenience, Lipschitz continuous curved sides on an element are constructed by connecting short lines, but the basis functions are defined on the whole curve. The convergence results for $P_1$ and $P_2$ elements are all optimal. Conclusions are drawn in Section \ref{con}.
% Let $\mathcal{T}_h$ be a partition of the domain $\Omega$ consisting of $C^{0,1}$ curvilinear polygons in two dimensions or polyhedrons
% in three dimensions. Denote by $\mathcal{E}_h$ the set of all edges or flat faces in $\mathcal{T}_h$, and let $\mathcal{E}_h^0 = \mathcal{E}_h/ \partial \Omega$ be
% the set of all interior edges or faces. For every element $D \in \mathcal{T}_h$ , we denote by
% $|D|$ the area or volume of $D$ and by $h_D$ its diameter. 
% We also set
% as usual the mesh size of $\mathcal{T}_h$ by
% $$h = \max\limits_{D \in \mathcal{T}_h} h_D.$$
% All the elements of $\mathcal{T}_h$ are assumed to be closed and simply connected $C^{0,1}$ curvilinear polygons
% or polyhedrons; see Figure \ref{fg1}. We need some shape regularity assumptions for the partition $\mathcal{T}_h$
% described below.
\section{Shape regularity}\label{shape-regular}
The shape regular assumptions are similar to \cite{Brenner17-2,guan2020w}. 
%Let $ {D}$ be the $C^{0,1}$ curvilinear polygonal (Figure \ref{fg1}) or polyhedral domain with diameter $h_D$. 
Assume

% \begin{align}\label{assume1}
	% &{D}\ {\rm is\ star\ shaped\ with\ respect\ to\ a\ disc/ball\ } \\
	% &\mathfrak{B}_D\subset D {\rm \ with\ radius\ =\ } \rho_D h_D,\ 0<\rho_D<1. \nonumber
	% \end{align}
\vspace{0.15 cm}
\begin{description}%[label=(\Alph*)]
	\item[(A1)] {\em $D$ is a $C^{0,1}$ curvilinear polygonal/polyhedral domain with diameter $h_D$,}
	\item[(A2)] {\em $D$ is star-shaped with respect to a disc/ball $\mathfrak{B}_D\subset D$ with radius $\rho_Dh_D,\ 0<\rho_D<\frac12$, }
	\item[(A3)] {\em $\rho_D$ has a uniform lower bound $0<\rho_{\rm min}<\rho_D$.}
\end{description}
\vspace{0.15 cm}
% The center of $\mathfrak{B}_D$ is the star center of $ {D}.$ Then we denote $\tilde{\mathfrak{B}}_D$ the  disc/ball concentric with $\mathfrak{B}_D$ whose radius is $h_D$. It's clear that 
% \begin{equation}\label{assume2}
	% \mathfrak{B}_D\subset D\subset \tilde{\mathfrak{B}}_D.
	% \end{equation}
Figure \ref{fg1} is a $C^{0,1}$ curvilinear polygonal domain, which has a Lipschitz continuous curved side.  
%% special cases for Lipschitz continuous curve: 1, piecewise $C^1$ curve, 2, connected short lines
Let $A \apprle B$ denote $A \leq (constant)B$.
% The notation $A \approx B$ is equivalent to $A \apprle B$ and $A \apprge B$.
% Figure \ref{fg1} is an example of $D$ satisfying the shape regularity assumptions. 
$D$ is shape-regular if it satisfies {\bf (A1)}-{\bf (A3)}. The following Lemmas in Section \ref{shape-regular} are valid on such $D$, and the hidden constants only depend on $\rho_D$ and degrees of employed polynomials.
\begin{lemma}\label{bramble}(Bramble-Hilbert Estimates) \cite{Bramble70}. $D$ is shape-regular implies that:
	\begin{equation*}
		\inf\limits_{q\in\mathbb{P}_l(D)} |\xi - q|_{H^m(D)} \apprle h^{l+1-m}
		|\xi|_{ H^{l+1}(D)}, \ 
		\forall \xi\in H^{l+1}(D),\ l = 0,\cdots, k,\ and\ 0\leq m \leq l.
	\end{equation*}
\end{lemma}
% Details can be found in \cite{Brenner07}, Lemma 4.3.8.

% \subsection{A Lipschitz Isomorphism between \texorpdfstring{$D$}{D} and \texorpdfstring{$\mathfrak{B}_D$}{Bd}}

% With the star-shaped assumptions {\bf (A1)}-{\bf (A2)}, there exists a  Lipschitz isomorphism
% $\Phi: \mathfrak{B}_D\rightarrow D$ such that both $|\Phi|_{W^{1,\infty}(\mathfrak{B}_D)}$ and $|\Phi|_{W^{1,\infty}(D)}$ are bounded by constant that only depends on $\rho_D$ (see \cite{V11}, Section 1.1.8).
% It then follows that 
% \begin{equation}\label{Dh}
	% |D|\approx h_D^n {\ \rm and \ } |\partial D|\approx h_D^{n-1},\ n=2,3,
	% \end{equation}
% where $|D|$ is the area of $D$ ($n=2$) or the volume of $D$ ($n=3$), and $|\partial D|$ is the arclength of $\partial D$ or the surface area of $D$ ($n=3$). Moreover from Theorem 4.1 in \cite{J87}, we have 
% \begin{eqnarray}
	% \|\xi\|_{L^2(\partial D)} &\approx& \|\xi\circ\Phi\|_{L^2(\partial \mathfrak{B}_D)},\quad \forall \xi\in L^2(\partial D) \label{iso1} \\
	% \|\xi\|_{L^2(  D)} &\approx& \|\xi\circ\Phi\|_{L^2( \mathfrak{B}_D)},\quad \forall \xi\in L^2(  D) \label{iso2}  \\
	% \|\xi\|_{H^1(D)} &\approx& \|\xi\circ\Phi\|_{H^1(  \mathfrak{B}_D)},\quad \forall \xi\in H^1(D) \label{iso3} 
	% \end{eqnarray}
% Same as in \cite{Brenner17-2}, from \eqref{Dh}, \eqref{iso1}-\eqref{iso3} and the standard (scaled) trace inequalities for $H^1(\mathfrak{B}_D)$ we have
\begin{lemma}\label{trace}
	(Trace Inequality (2.18) ) \cite{Brenner17-2}. 
	If $D$ is shape-regular, then we have
	$$
	h_D^{-1}\|\xi\|_{L^2(\partial D)}^2 \apprle  h_D^{-2}\|\xi\|_{L^2(D)}^2 + \|\nabla \xi \|_{L^2(D)}^2, \ \forall \xi \in H^1(D ).
	$$
\end{lemma}
{ 
\begin{remark}
    (2.18) in \cite{Brenner17-2} is also valid for $D$ here with curved sides/faces. The proof is the same.
\end{remark}
}
% \subsection{ \texorpdfstring{$L^2$}{L} Projection Operators}
\subsection{ $L^2$ Projection Operators}
% On $D$, let $Q_{k,D}^0$ be the $L^2$ projection operator from $L^2 (D)$ to $\mathbb{P}_k (D)$.
% Analogously, for each edge or face (flat or not) $e$, let $Q_{k, D}^b$ be the $L^2$ projection operator
% from $L^2 (e)$ to $\mathbb{P}_k|_e$. 
{  We define the projection $Q_h$ as 
	\begin{equation}\label{l2proj}
		Q_h v|_D := (Q_{k,D}^0 v_0 , Q_{k,D}^b v_b),\
		\forall v \in W(D),
	\end{equation}
	where $Q_{k,D}^0$ is the $L^2$ projection operator from $L^2 (D )$ to $\mathbb{P}_k (D);$
	$Q_{k,D}^b$ is the $L^2$ projection operator
	from $L^2 (e)$ to $\mathbb{P}_k|_e$, 
	where $e$ is a side or face of $D$, no need to be straight or flat, and $\mathbb{P}_k|_e$ is the space of restricted parts of polynomials with degrees no more than $k$ on $e$.\\
	Then, let $\mathbb{Q}_{k-1,D}$ be the $L^2$ projection operator from $[L^2(D)]^n$ to  $[\mathbb{P}_{k-1} (D)]^n, n= 2,3.$   }

With shape-regular $D$ and definitions of $Q_h$ and $\mathbb{Q}_{k-1, D}$, we obtain Lemmas \ref{l4} to \ref{leq}. The proof of Lemma \ref{l4} can be found in \cite{Brenner17-2} (see Lemma 2.3 and Lemma 3.9).  Proofs of Lemmas \ref{l_discrete} and \ref{ler} (Lemma \ref{ler} depends on Lemma \ref{bramble} and \eqref{QkD0}) are similar to Lemmas 3 and 6 in \cite{guan2020w}, respectively. Lemma \ref{leq} is crucial for the error analysis of our method.
\begin{lemma}\cite{Brenner17-2}\label{l4} Assume that $D$ is shape-regular. Then, we have
	\begin{align}
		&|p|_{H^1(D)} \apprle h_D^{-1}\| p \|_{L^2(D)},\ \forall p\in \mathbb{P}_k(D) \label{l44}\\
		&|Q_{k,D}^0\xi|_{H^1(D)} \apprle |\xi|_{H^1(D)},\ \forall \xi\in H^1(D) \label{QkD0}
	\end{align}
	% the hidden constant only depends on $\rho_D$ and $k$.
\end{lemma}
% \begin{proof}
	% 	The proof is similar to	Lemma 2.3 in \cite{Brenner17-2}.
	% \end{proof}
% \begin{lemma}\cite{Brenner17-2}\label{QkD0} Assume that $D$ is shape-regular. Then, we have
	% 	$$
	% 	|Q_{k,D}^0\xi|_{H^1(D)} \apprle |\xi|_{H^1(D)},\ \forall \xi\in H^1(D).
	% 	$$
	% 	% the hidden constant only depends on $\rho_D$ and $k$.
	% \end{lemma}
\begin{lemma}\cite{guan2020w}\label{l_discrete}
	For any $\vq\in [\mathbb{P}_k(D)]^n$, $D$ is shape-regular, we have
	$$
	h_D\|\vq\|^2_{L^2(\partial D)}+h_D^{2}\|\nabla\cdot \vq\|^2_{L^2(D)}
	\apprle 
	\|\vq\|^2_{L^2(D)}.
	$$	
	% the hidden constant only depends on $\rho_D$ and $k$.	
\end{lemma}
% \begin{proof}
	% 	Suppose $n=2$, and $\vq = (q_1,q_2)$, then by Lemma \ref{trace} and Lemma \ref{l4}, we have
	% 	$$
	% 	h_D\|q_i\|^2_{L^2(\partial D)}+h_D^{2} | q_i|^2_{H^1(D)}
	% 	\apprle 
	% 	\|q_i\|^2_{L^2(D)},\ i = 1,2,
	% 	$$
	% 	so that 
	% 	$$
	% 	h_D\|\vq\|^2_{L^2(\partial D)}+h_D^{2}\|\nabla\cdot \vq\|^2_{L^2(D)}
	% 	\apprle 
	% 	\|\vq\|^2_{L^2(D)}.
	% 	$$	
	% 	For $n=3,$ the proof is similar.
	% \end{proof}
% The following lemma provides some estimates for the projection operators $Q_h$ and $\mathbb{Q}_{k-1,D}.$
\begin{lemma}\cite{guan2020w}\label{ler}
	Let $D$ be shape-regular, then for $\xi\in H^{k+1}(D)$, we have
	\begin{align}
		&\|\xi-Q_{k,D}^0 \xi\|_{L^2(D)}^2+
		h_D^2 |\xi-Q_{k,D}^0 \xi |_{H^1(D)}^2
		\apprle  h_D^{2(k+1)}\|\xi\|_{H^{k+1}(D)}^2, \label{l2Q1} \\ 
		&\|\nabla\xi - \mathbb{Q}_{k-1,D}\nabla\xi
		\|_{L^2(D)}^2
		+
		h_D^2|\nabla\xi - \mathbb{Q}_{k-1,D}\nabla\xi|_{H^1(D)}^2
		\apprle h_D^{2k}\|\xi\|_{H^{k+1}(D)}^2. \label{l2Q2}
	\end{align}
	% the hidden constant only depends on $\rho_D$ and $k$.
\end{lemma}
\begin{lemma}\label{leq}
	Let $Q_h$ be the operator in \eqref{l2proj}. Then with shape-regular $D$,  we have
	\begin{align}
		&(\nabla_{w} Q_h \xi ,\vq)_D = (\mathbb{Q}_{k-1,D}  \nabla\xi ,\vq)_D 
		+ 
		\langle  
		Q^b_{k,D}\xi -\xi, 
		\vq \cdot\vn
		\rangle_{\partial D},\
		\forall  \xi\in H^1(D), \label{QhQ1}
		\\
		&\left|
		\langle  
		Q^b_{k,D}\xi -\xi, 
		\vq \cdot\vn
		\rangle_{\partial D}
		\right|
		\apprle
		h_D^k\|\xi\|_{H^{k+1}(D)}\|\vq\|_{L^2(D)},\
		\forall  \xi\in H^{k+1}(D), \label{QhQ2} 
		\\
		&\|\nabla_{w} Q_h \xi -\mathbb{Q}_{k-1,D}  \nabla\xi \|_{L^2(D)} 
		\apprle
		h_D^k\|\xi\|_{H^{k+1}(D)},\
		\forall  \xi\in H^{k+1}(D), \label{QhQ3} 
	\end{align}
	where $\vq,\nabla_{w} Q_h \xi \in [\mathbb{P}_{k-1} (D)]^n$, the hidden constants only depend on $\rho_D$ and $k$.
\end{lemma}
\begin{proof}
	By \eqref{wg2}, integration by parts and the definitions of $\mathbb{Q}_{k-1,D}$, $Q_h$, we have
	\begin{eqnarray*}
		(\nabla_{w} Q_h \xi ,\vq)_D 
		&=& 
		-(Q^0_{k,D} \xi ,\nabla\cdot\vq)_D + 
		\langle  
		Q^b_{k,D}\xi, 
		\vq \cdot\vn
		\rangle_{\partial D}\\
		&=& 
		-(\xi ,\nabla\cdot\vq)_D 
		+ 
		\langle  
		\xi, 
		\vq \cdot\vn
		\rangle_{\partial D}
		+ 
		\langle  
		Q^b_{k,D}\xi - \xi, 
		\vq \cdot\vn
		\rangle_{\partial D}\\
		&=& 
		(\mathbb{Q}_{k-1,D}  \nabla\xi ,\vq)_D + 
		\langle  
		Q^b_{k,D}\xi -\xi, 
		\vq \cdot\vn
		\rangle_{\partial D}
	\end{eqnarray*}
	so that \eqref{QhQ1} is obtained.
	
	To get \eqref{QhQ2},  with Lemma \ref{l_discrete}, we have
	\begin{eqnarray}\label{Q_b1}
		\left|
		\langle  
		Q^b_{k,D}\xi -\xi, 
		\vq \cdot\vn
		\rangle_{\partial D}
		\right|
		&\apprle&
		\|Q^b_{k,D}\xi -\xi\|_{L^2(\partial D)}\|\vq \|_{L^2(\partial D)} \nonumber \\
		&\apprle&
		h_D^{-\frac12}\|Q^b_{k,D}\xi -\xi\|_{L^2(\partial D)}\|\vq \|_{L^2( D)}
	\end{eqnarray}
	then let $p\in \mathbb{P}_{k}(D)$,
	\begin{eqnarray*}\label{Q_bxi}
		\|Q^b_{k,D}\xi -\xi\|_{L^2(\partial D)}
		&\apprle&
		% \|Q^b_{k,D}\xi -p\|_{L^2(\partial D)}
		% +
		% \|\xi-p\|_{L^2(\partial D)}\\
		% &\apprle&
		% \|Q^b_{k,D}(\xi -p)\|_{L^2(\partial D)}
		% +
		% \|\xi-p\|_{L^2(\partial D)}\\
		% &\apprle&
		\|\xi-p\|_{L^2(\partial D)}
	\end{eqnarray*}
	so that  with Lemma \ref{trace} and inequality \eqref{l2Q1}, let $p$ be $Q_{k,D}^0\xi$, we have
	\begin{equation}\label{Q_b2}
		h_D^{-\frac12}\|Q^b_{k,D}\xi -\xi\|_{L^2(\partial D)}\apprle h^k\|\xi\|_{H^{k+1}(D)},
	\end{equation}
	with \eqref{Q_b1} and \eqref{Q_b2}, we get  \eqref{QhQ2}.
	
	To get \eqref{QhQ3}, from \eqref{QhQ1}, we have
	\begin{eqnarray*}
		(\nabla_{w} Q_h \xi -\mathbb{Q}_{k-1,D}\nabla\xi,\vq)_D 
		=
		\langle  
		Q^b_{k,D}\xi -\xi, 
		\vq \cdot\vn
		\rangle_{\partial D}
	\end{eqnarray*}
	let $\vq = \nabla_{w} Q_h \xi -\mathbb{Q}_{k-1,D}\nabla\xi$, with \eqref{Q_b1}
	\begin{eqnarray*}
		\|\nabla_{w} Q_h \xi -\mathbb{Q}_{k-1,D}\nabla\xi\|^2_{L^2(D)} 
		&\apprle&
		\langle  
		Q^b_{k,D}\xi -\xi, 
		\vq \cdot\vn
		\rangle_{\partial D} \\
		&\apprle&
		h_D^{-\frac12}\|Q^b_{k,D}\xi -\xi\|_{L^2(\partial D)}
		\|\nabla_{w} Q_h \xi -\mathbb{Q}_{k-1,D}\nabla\xi\|_{L^2(D)} 
	\end{eqnarray*}
	with \eqref{Q_b2}, we get  \eqref{QhQ3}.
\end{proof}
\begin{remark}
	If edge/face $e \subset \partial D$ is part of a line/plane, then 
	$\langle  
	Q^b_{k,D}\xi -\xi, 
	\vq \cdot\vn
	\rangle_{e} = 0$.
\end{remark}
\subsection{The Weak Galerkin Finite Element Scheme}\label{wgfem-scheme}
Let $\Omega$ be a bounded domain (2D or 3D) with $C^{0,1}$ boundary or interface. Suppose $\mathcal{T}_h$ is the partition of $\Omega$, each element $D$ of $\mathcal{T}_h$ is shape-regular, and $\mathcal{E}_h$ is the set of sides/faces in $\mathcal{T}_h$, $h = \max\limits_{D \in \mathcal{T}_h} h_D.$ 
% For each $D\in \mathcal{T}_h,$ we have $W(D)$ defined in \eqref{wd}. Then let $W$ be the weak functional space on $\mathcal{T}_h$ as
% $$
% W := \prod_{D\in \mathcal{T}_h} W(D).
% $$
% Same as Section 4.2 in \cite{Lin15}, we denote $V$ as a subspace of $W$. For each interior edge $e \in \mathcal{E}_h^0$, there are $D_1$ and $D_2$, so that $e\subset \partial D_1\cap \partial D_2$. Denote $v\in V$, so that for $v_i\in W(D_i), i=1,2,$  we have 
% $$
% v_1|_e = v_2|_e.
% $$
\begin{figure}[H]
	\begin{center}
		\begin{tikzpicture}[scale = 1.5]
			\draw (0,0) rectangle (4,2);
			%        \draw (1.5,0) parabola (2.5,2);
			\draw (1.8,0) .. controls (2,1) and (2,0) .. (2.3,2);
			\node [ ] at (2,1 ) {$e$};
			\node [ ] at (1,1 ) {$D_1$};
			\node [ ] at (3,1 ) {$D_2$};
		\end{tikzpicture}
	\end{center}\caption{$e$ is shared by two elements}\label{fg-e}
\end{figure}
Let $D_0$ be the inner part of $D$, $\mathbb{P}_k(D_0)$ be the space of polynomials on $D_0$ with degrees no more than $k$, and on each side/face, $e\in \mathcal{E}_h,$ let $\mathbb{P}_k|_e$ be the space of traces of polynomials $\mathbb{P}_k(\mathbb{R}^n)$ on $e$, so the basis functions on shared $e$, see Figure \ref{fg-e}, are uniquely determined by collecting the linearly independent traces of $\mathbb{P}_k(\mathbb{R}^n)$. For example, traces of $\mathbb{P}_1(\mathbb{R}^2)$ on $e$, in Figure \ref{fg-e}, are $\{1,x,y\}$. Then the weak Galerkin finite element space is given by \eqref{Sjl}
\begin{equation}\label{Sjl}
	V_h:= \{v: v|_{D_0}\in \mathbb{P}_k(D_0)\ \forall D\in \mathcal{T}_h \ {\rm and}\ v|_{e}\in \mathbb{P}_k|_e \ \forall e\in \mathcal{E}_h \}.
\end{equation}
Let the space $V^0_h$ be the subspace of $V_h$ which has vanishing boundary value on $\partial\Omega$ 
\begin{equation}\label{Sjl0}
	V^0_h := \{v: v\in V_h \ {\rm and}\ v|_{\partial\Omega} =0  \}.
\end{equation}
We then define $V_h|_D$ as the space $\{v=(v_0,v_b): v_0\in \mathbb{P}_k(D_0)\ {\rm and}\ v_b\in \mathbb{P}_k|_e\ \forall e\in \partial D  \}.$
\begin{lemma}\label{lem_h1}
	If $D$ is shape-regular, for any $\xi \in H^{k+1} (D)$ and
	$v \in V_h|_D$, we have
	\begin{align}
		&\left|
		h_D^{-1}
		\langle  
		Q_{k,D}^0 \xi-Q_{k,D}^b \xi ,  
		v_0-v_b
		\rangle_{\partial D}
		\right|\apprle h_D^k\|\xi\|_{H^{k+1}(D)}
		h_D^{-\frac12}
		\| v_0-v_b \|_{L^2(\partial D)}, \label{e-trace1}
		\\
		&\left|
		\langle  
		(\nabla \xi - \mathbb{Q}_{k-1,D} \nabla \xi)\cdot\vn,   
		v_0-v_b
		\rangle_{\partial D}
		\right|\apprle   h_D^k\|\xi\|_{H^{k+1}(D)}
		h_D^{-\frac12}
		\| v_0-v_b \|_{L^2(\partial D)}, \label{e-trace2}
	\end{align}
	where $k\geq 1$ and the hidden constant only depends on $\rho_D$ and $k$. 
	% Note that the norm is an $ H^1$ equivalence for finite element functions with vanishing boundary values.
\end{lemma}
\begin{proof}
	To get \eqref{e-trace1}, we have
	\begin{eqnarray*}
		\left|
		h_D^{-1}
		\langle  
		Q_{k,D}^0 \xi-Q_{k,D}^b \xi , 
		v_0-v_b
		\rangle_{\partial D}
		\right|
		&\apprle& 
		h_D^{-\frac12}
		\| Q_{k,D}^0 \xi- Q_{k,D}^b \xi \|_{L^2(\partial D)}
		h_D^{-\frac12}
		\| v_0-v_b \|_{L^2(\partial D)},
	\end{eqnarray*}
	for the first term on the right side, we have  
	$$
	h_D^{-\frac12}
	\| Q_{k,D}^0 \xi- Q_{k,D}^b \xi \|_{L^2(\partial D)}
	\apprle
	h_D^{-\frac12}
	\| Q_{k,D}^0 \xi- \xi \|_{L^2(\partial D)}
	+
	h_D^{-\frac12}
	\| Q_{k,D}^b \xi- \xi \|_{L^2(\partial D)},
	$$
	with Lemma \ref{trace}, Lemma \ref{ler} and \eqref{Q_b2}, then \eqref{e-trace1} is obtained.
	
	To get \eqref{e-trace2}, we have
	\begin{eqnarray*}
		\left|
		\langle  
		(\nabla \xi - \mathbb{Q}_{k-1,D} \nabla \xi)\cdot\vn,   
		v_0-v_b
		\rangle_{\partial D}
		\right|
		\apprle
		h_D^{\frac12} 
		\|\nabla \xi - \mathbb{Q}_{k-1,D} \nabla \xi\|_{L^2(\partial D)} 
		h_D^{-\frac12}
		\| v_0-v_b \|_{L^2(\partial D)},
	\end{eqnarray*}
	with Lemma \ref{trace} and Lemma \ref{ler}, \eqref{e-trace2} is obtained.
\end{proof}
\begin{lemma}\cite{guan2020w}\label{K-bound}
	Assume that $D$ is shape-regular. Then we have
	\begin{equation}
		\|\nabla v_0\|^2_{L^2(D)}
		\apprle
		\|\nabla_{w} v\|^2_{L^2(D)}+
		h_D^{-1}\|v_b-v_0\|^2_{L^2(\partial D)},\  \forall v\in V_h|_{D},
	\end{equation}
	the hidden constant only depends on $\rho_D$ and $k$.
\end{lemma}
% \begin{proof}
	% 	Suppose on $D$ we have $v=(v_0,v_b)_D$, then by the definition of $\nabla_{w}$, we have
	% 	\begin{eqnarray*}
		% 		(\nabla_{w} v,\vq )_D 
		% 		&=& 
		% 		-(v_0,\nabla\cdot\vq)_D +\langle v_b,\vq\cdot\vn \rangle_{\partial D}\\
		% 		&=& 
		% 		(\nabla v_0,\vq)_D +\langle v_b-v_0,\vq\cdot\vn \rangle_{\partial D}
		% 	\end{eqnarray*}
	% 	so that
	% 	$$
	% 	(\nabla_{w} v-\nabla v_0,\vq )_D = \langle v_b-v_0,\vq\cdot\vn \rangle_{\partial D}
	% 	$$
	% 	let $\vq = \nabla_{w} v-\nabla v_0$, then
	% 	$$
	% 	(\nabla_{w} v-\nabla v_0,\nabla_{w} v-\nabla v_0 )_D = \langle v_b-v_0,(\nabla_{w} v-\nabla v_0)\cdot\vn \rangle_{\partial D},
	% 	$$
	% 	By the discrete inequalities of polynomials, Lemma \ref{l_discrete}, we have
	% 	$$
	% 	\|\nabla_{w} v-\nabla v_0\|_{L^2(D)}\apprle
	% 	h_D^{-\frac12}\|v_b-v_0\|_{L^2(\partial D)},
	% 	$$
	% 	so that
	% 	\begin{equation*}
		% 	\|\nabla v_0\|^2_{L^2(D)}
		% 	\apprle
		% 	\|\nabla_{w} v\|^2_{L^2(D)}+
		% 	h_D^{-1}\|v_b-v_0\|^2_{L^2(\partial D)}.
		% 	\end{equation*}
	% \end{proof}
\section{The weak Galerkin finite element method for elliptic equation}\label{poisson}
Let $\Omega$ be a bounded domain with $C^{0,1}$ boundary in $\mathbb{R}^n$, $f\in L^2(\Omega)$, the Poisson's equation is
\begin{equation}\label{poisson-equation}
	\begin{cases}
		\ -\Delta u            &= f, \\
		\ u|_{\partial \Omega} &= 0.
	\end{cases}
\end{equation}
For any $v\in V_h,$ the weak gradient of $v$ is defined on each element $D$ by \eqref{wg2},   respectively. 
For any $u,v\in V_h$, the bilinear form is defined as
\begin{equation}\label{bhr} 
	a_{h}(u,v) = \sum\limits_{D\in\mathcal{T}_h}\int_{D}\nabla_{w} u\cdot\nabla_{w} v\ {\rm d} x.
\end{equation}
The stabilization term is:
\begin{equation}\label{stable} 
	s_{h}(u,v) = \sum\limits_{D\in\mathcal{T}_h} 
	h_D^{-1}
	\langle  
	u_0-u_b, v_0-v_b
	\rangle_{\partial D}.
\end{equation}
A numerical solution for \eqref{poisson-equation} can be obtained by seeking $u_h =(u_0,u_b)\in V^0_h$ such that 
\begin{equation}\label{num-poisson}
	a_s(u_h,v) := a_h(u_h,v)+s_h(u_h,v) = (f,v_0)_{\Omega}, \ \forall v=(v_0,v_b)\in V_h^0.
\end{equation}
Then the weak-1 norm of $v\in V$ is defined as
\begin{equation}\label{w1n}
	| v |^2_{k-1,w} = \sum\limits_{D\in\mathcal{T}_h}
	\|\nabla_{w} v \|_{L^2(D)}^2
	+ 
	h_D^{-1}
	\| 
	v_0-v_b 
	\|_{L^2(\partial D)}^2,
	%\right).
\end{equation}
where $k\geq 1$ is an integer.\\
We generate Figure \ref{fg-horn} by PolyMesher \cite{talischi2012polymesher} to show how the mesh in 2D could be.
\begin{figure}[H]
	\begin{center}
		\includegraphics[trim={0 3cm 0 2cm},clip, width=0.5\linewidth]{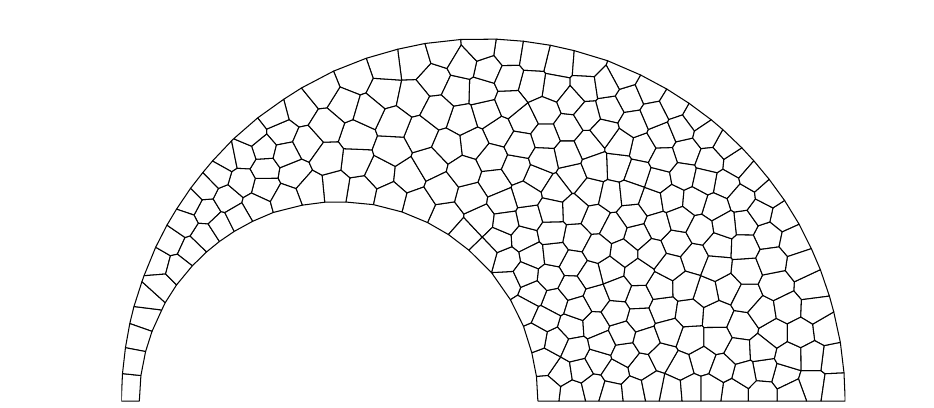}
	\end{center}\caption{A shape regular partition of domain $\Omega$}\label{fg-horn}
\end{figure}
\begin{lemma}\label{poincare}
	Suppose the partition $\mathcal{T}_h$ is shape-regular. Then we have
	$$
	\|v_0\|_{L^2(\Omega)} \apprle |v|_{k-1,w}, \ \forall v=(v_0,v_b)\in V_h^0,
	$$
	the hidden constant only depends on $\rho_D$ and $k$.
\end{lemma}
\begin{proof} 
	The key to proving Lemma \ref{poincare} is that on each edge or face $e$, $v|_e$ is unique. 
	Then, same as Lemma 7.1 in \cite{Lin15}, we have the discrete Poincar$\acute{\rm e}$ inequality.
\end{proof}
Also, we have the existence and uniqueness of the solution of \eqref{num-poisson}.
\subsection{Error Analysis}
Let $u \in H^2(\Omega)$ be the solution of \eqref{poisson-equation} and $v\in V_h^0$. 
Then, multiply \eqref{poisson-equation} by $v_0$ of $v=(v_0,v_b)\in V_h^0$ we have
\begin{equation}\label{wu}
	\sum\limits_{D\in\mathcal{T}_h} (\nabla u,\nabla v_0)_D
	=(f,v_0)_{\Omega}+\sum\limits_{D\in\mathcal{T}_h}\langle  v_0-v_b , \nabla u\cdot\vn  \rangle_{\partial D},
\end{equation}
where 
$\sum\limits_{D\in\mathcal{T}_h}\langle  v_b , \nabla u\cdot\vn  \rangle_{\partial D} = 0 .$\\
It follows from \eqref{wg2}, \eqref{QhQ1} and the integration by parts 
\begin{align}\label{wI_h}
	(\nabla_{w} & Q_hu,\nabla_{w}v)_D \nonumber \\
	&= (\mathbb{Q}_{k-1,D}\nabla u, \nabla_{w} v)_D 
	+ 
	\langle  
	Q^b_{k,D}u -u, 
	\nabla_{w}v \cdot\vn
	\rangle_{\partial D}  \nonumber \\
	&= -(v_0,\nabla\cdot(\mathbb{Q}_{k-1,D}\nabla u))_D
	+
	\langle v_b,(\mathbb{Q}_{k-1,D}\nabla u)\cdot\vn \rangle_{\partial D}
	+ 
	\langle  
	Q^b_{k,D}u -u, 
	\nabla_{w}v \cdot\vn
	\rangle_{\partial D} 
	\\ 
	&= (\nabla v_0,\mathbb{Q}_{k-1,D}\nabla u)_D
	-
	\langle v_0-v_b,(\mathbb{Q}_{k-1,D}\nabla u)\cdot\vn \rangle_{\partial D} 
	+ 
	\langle  
	Q^b_{k,D}u -u, 
	\nabla_{w}v \cdot\vn
	\rangle_{\partial D} 
	\nonumber 
	% \displaybreak[4] 
	\\
	&= (\nabla u,\nabla v_0)_D
	-
	\langle v_0-v_b,(\mathbb{Q}_{k-1,D}\nabla u)\cdot\vn \rangle_{\partial D}
	+ 
	\langle  
	Q^b_{k,D}u -u, 
	\nabla_{w}v \cdot\vn
	\rangle_{\partial D}. \nonumber
\end{align}
% 	\displaybreak[4]
Combining \eqref{wu} and \eqref{wI_h}, we have
\begin{align}\label{w-error1}
	\sum\limits_{D\in\mathcal{T}_h} 
	(\nabla_{w} Q_hu,\nabla_{w}v)_D
	=\ &
	(f,v_0)_{\Omega} 
	+
	\sum\limits_{D\in\mathcal{T}_h}
	\langle  
	v_0-v_b , 
	(\nabla u - \mathbb{Q}_{k-1,D} \nabla u)\cdot\vn  
	\rangle_{\partial D} 
	\\
	&+
	\sum\limits_{D\in\mathcal{T}_h}
	\langle  
	Q^b_{k,D}u -u, 
	\nabla_{w}v \cdot\vn
	\rangle_{\partial D}. \nonumber
\end{align}
Adding $s_h({Q}_h u, v)$ to both sides of \eqref{w-error1} gives
\begin{align}\label{w-error2}
	a_s(Q_h u,v)
	=\ &
	(f,v_0)_{\Omega} 
	+
	\sum\limits_{D\in\mathcal{T}_h}
	\langle  
	v_0-v_b , 
	(\nabla u - \mathbb{Q}_{k-1,D} \nabla u)\cdot\vn  
	\rangle_{\partial D}  
	\\
	&+
	\sum\limits_{D\in\mathcal{T}_h}
	\langle  
	Q^b_{k,D}u -u, 
	\nabla_{w}v \cdot\vn
	\rangle_{\partial D}
	+s_h(Q_hu,v). \nonumber
\end{align}
Subtracting \eqref{num-poisson} from \eqref{w-error2}, we have the error equation
\begin{align}\label{w-error3}
	a_s(e_h,v)
	=&
	\sum\limits_{D\in\mathcal{T}_h}
	\langle  
	v_0-v_b , 
	(\nabla u - \mathbb{Q}_{k-1,D} \nabla u)\cdot\vn  
	\rangle_{\partial D} \\
	&+
	\sum\limits_{D\in\mathcal{T}_h}
	\langle  
	Q^b_{k,D}u -u, 
	\nabla_{w}v \cdot\vn
	\rangle_{\partial D}
	+s_h(Q_hu,v). \nonumber
\end{align}
where
$$
e_h|_D = (e_0,e_b)_D:= (Q_{k,D}^0 u -u_0, Q_{k,D}^b u -u_b)_D= (Q_hu-u_h)|_{D}
$$
which is the error between the weak Galerkin finite element solution $(u_0,u_b)$ and the $L^2$ projection of the exact solution.
Then we define a norm $\|\cdot\|_h$ as
$$
\|\vv\|_h^2 := \sum\limits_{D\in\mathcal{T}_h} \|\vv\|_{L^2(D)}, \ \forall \vv\in [L^2(\Omega)]^n.
$$
\begin{theorem}
	Let $u_h \in V_h^0$ be the weak Galerkin finite element solution of the
	problem \eqref{poisson-equation}. Assume that the exact solution is so regular that
	$u \in H^{k+1} (\Omega)$. Then we have
	\begin{eqnarray}
		\|\nabla u - \nabla_{w} u_h\|_h 
		&\apprle& h^k\|u\|_{H^{k+1}(\Omega)}, \label{we1} \\
		\|\nabla u - \nabla u_0\|_h 
		&\apprle& 
		h^k\|u\|_{H^{k+1}(\Omega)}, \label{we2} 
	\end{eqnarray}
	the hidden constants only depend on $\rho_D$ and $k$.
\end{theorem}
\begin{proof}
	Let $v = e_h$ in \eqref{w-error3}, we have
	\begin{align*}
		|e_h|_{k-1,w}^2
		=&
		\sum\limits_{D\in\mathcal{T}_h}
		\langle  
		e_0-e_b , 
		(\nabla u - \mathbb{Q}_{k-1,D} \nabla u)\cdot\vn  
		\rangle_{\partial D} \\
		&+
		\sum\limits_{D\in\mathcal{T}_h}
		\langle  
		Q^b_{k,D}u -u, 
		\nabla_{w}e_h \cdot\vn
		\rangle_{\partial D}
		+s_h(Q_hu,e_h).
	\end{align*}
	It then follows from \eqref{QhQ2} and Lemma \ref{lem_h1} 
	\begin{equation}\label{error1}
		|e_h|_{k-1,w}^2 \apprle  h^k\|u\|_{H^{k+1}(\Omega)}|e_h|_{k-1,w}.
	\end{equation}
	Based on \eqref{error1}, firstly, we prove \eqref{we1}, 
	\begin{eqnarray*}
		\|\nabla u - \nabla_{w} u_h\|_h 
		&\leq&
		\|\nabla u - \mathbb{Q}_{k-1}(\nabla u)\|_h 
		+\|\mathbb{Q}_{k-1}(\nabla u) - \nabla_{w}Q_h u\|_h
		+\|\nabla_{w}Q_h u - \nabla_{w} u_h\|_h, 
	\end{eqnarray*}
	with Lemma \ref{ler} and Lemma \ref{leq} and
	$$
	\|\nabla_{w}(Q_h u - u_h)\|_h\leq |e_h|_{k-1,w}
	$$
	we have
	\eqref{we1}.
	
	Secondly, with Lemma \ref{K-bound}, we have
	\begin{eqnarray*}
		\sum_{D\in\mathcal{T}_h}\|\nabla (Q^0_{k,D} u -u_h|_{D_0})\|^2_{L^2(D)} 
		&=&
		\sum\limits_{D\in\mathcal{T}_h} \|\nabla  e_0\|^2_{L^2(D)}\\
		&\apprle&
		\sum\limits_{D\in\mathcal{T}_h} \|\nabla_{w} e_h\|^2_{L^2(D)} +h_D^{-1}\|e_b-e_0\|^2_{L^2(\partial D)}\\
		&\apprle&
		|u_h-Q_h u|_{k-1,w}^2
	\end{eqnarray*}
	which means
	$$
	\sum_{D\in\mathcal{T}_h}\|\nabla (Q^0_{k,D} u -u_h|_{D_0})\|^2_{L^2(D)}  \apprle h^{2k}\|u\|^2_{H^{k+1}(\Omega)}.
	$$
	Also by Lemma \ref{ler}
	$$
	\sum_{D\in\mathcal{T}_h}\|\nabla (Q^0_{k,D} u -u)\|^2_{L^2(D)}  \apprle h^{2k}\|u\|^2_{H^{k+1}(\Omega)},
	$$
	then we have \eqref{we2}
	$$
	\|\nabla u - \nabla u_0\|_h  \apprle h^k\|u\|_{H^{k+1}(\Omega)}.
	$$
\end{proof}
\begin{theorem}
	Let $u_h \in V_h^0$ be the weak Galerkin finite element solution of the
	problem \eqref{poisson-equation}. Assume that the exact solution is so regular that
	$u \in H^{k+1} (\Omega)$. Then we have
	\begin{eqnarray}\label{l2e}
		\|u - u_0\|_{L^2(\Omega)}
		&\apprle& h^{k+1}\|u\|_{H^{k+1}(\Omega)},
	\end{eqnarray}
	the hidden constant only depends on $\rho_D$ and $k$.
\end{theorem}
\begin{proof}
	We begin with a dual problem seeking $\phi\in H_0^2(\Omega)$ such that
	$
	-\Delta \phi = e_0.
	$
	Suppose we have $\|\phi\|_{H^2(\Omega)}\apprle \|e_0\|_{L^2(\Omega)}$.
	
	Then we have
	\begin{equation}\label{l2e1}
		\|e_0\|_{L^2(\Omega)}^2 = 
		\sum\limits_{D\in\mathcal{T}_h}(\nabla\phi,\nabla e_0)_D 
		-
		\sum\limits_{D\in\mathcal{T}_h}
		\langle 
		\nabla\phi\cdot\vn, 
		e_0-e_b 
		\rangle_{\partial D}.
	\end{equation}
	Let $u=\phi$ and $v=e_h$ in \eqref{wI_h}, we have
	\begin{equation}\label{l2e2}
		(\nabla \phi,\nabla e_0)_D	
		=
		(\nabla_{w} Q_h \phi,\nabla_{w} e_h)_D
		+
		\langle e_0-e_b,(\mathbb{Q}_{k-1,D}\nabla \phi)\cdot\vn \rangle_{\partial D}
		-
		\langle  
		Q^b_{k,D}\phi -\phi, 
		\nabla_{w}e_h \cdot\vn
		\rangle_{\partial D}. 
	\end{equation}
	Combining \eqref{l2e1} and \eqref{l2e2}, we have
	\begin{align}\label{l2e3}
		\|e_0\|_{L^2(\Omega)}^2 =\ & 
		(\nabla_{w} Q_h \phi,\nabla_{w} e_h)_\Omega
		+
		\sum\limits_{D\in\mathcal{T}_h}
		\langle 
		(\mathbb{Q}_{k-1,D}\nabla \phi-\nabla\phi)\cdot\vn, 
		e_0-e_b 
		\rangle_{\partial D} 
		\\
		& - \sum\limits_{D\in\mathcal{T}_h}
		\langle  
		Q^b_{k,D}\phi -\phi, 
		\nabla_{w}e_h \cdot\vn
		\rangle_{\partial D}. \nonumber
	\end{align}
	So that by Lemma \ref{leq} and Lemma \ref{lem_h1}, we have
	\begin{align}
		&\left|
		\sum\limits_{D\in\mathcal{T}_h}
		\langle 
		(\mathbb{Q}_{k-1,D}\nabla \phi-\nabla\phi)\cdot\vn, 
		e_0-e_b 
		\rangle_{\partial D}
		\right|
		\apprle	
		h\|\phi\|_{H^{2}(\Omega)}|e_h|_{k-1,w}, \label{term2} 
		\\
		&\left|	
		\sum\limits_{D\in\mathcal{T}_h}
		\langle  
		Q^b_{k,D}\phi -\phi, 
		\nabla_{w}e_h \cdot\vn
		\rangle_{\partial D}
		\right|
		\apprle	
		h\|\phi\|_{H^{2}(\Omega)}|e_h|_{k-1,w}.\label{term3}
	\end{align}	
	Then let $v = Q_h \phi$ in \eqref{w-error3}, such that
	\begin{eqnarray*}
		(\nabla_{w} Q_h \phi,\nabla_{w} e_h)_\Omega 
		&=&
		\sum\limits_{D\in\mathcal{T}_h}
		\langle  
		Q^0_{k,D}\phi- Q^b_{k,D}\phi , 
		(\nabla u - \mathbb{Q}_{k-1,D} \nabla u)\cdot\vn  
		\rangle_{\partial D} \label{term11}\\
		&&+
		\sum\limits_{D\in\mathcal{T}_h}
		\langle  
		Q^b_{k,D}u -u, 
		(\nabla_{w} Q_h \phi -\nabla \phi)\cdot\vn
		\rangle_{\partial D} \label{term12}
		\\
		&&+s_h(Q_hu,Q_h \phi) -s_h(e_h,Q_h \phi), \label{term13}
	\end{eqnarray*}
	where
	$$
	\sum\limits_{D\in\mathcal{T}_h}
	\langle  
	Q^b_{k,D}u -u, 
	\nabla \phi \cdot\vn
	\rangle_{\partial D} = 0.
	$$
	Same as the proof of Theorem 8.2 in \cite{Lin15}, we have
	$$
	\sum\limits_{D\in\mathcal{T}_h}
	\langle  
	Q^0_{k,D}\phi- Q^b_{k,D}\phi , 
	(\nabla u - \mathbb{Q}_{k-1,D} \nabla u)\cdot\vn  
	\rangle_{\partial D}
	\apprle
	h^{k+1}\|u\|_{H^{k+1}(\Omega)}\|\phi\|_{H^2(\Omega)}
	$$
	and
	$$
	|s_h(Q_hu,Q_h \phi)| +|s_h(e_h,Q_h \phi)|\apprle h^{k+1}\|u\|_{H^{k+1}(\Omega)}\|\phi\|_{H^2(\Omega)}.
	$$	
	Then
	\begin{eqnarray*}
		|\langle  
		Q^b_{k,D}u -u, 
		(\nabla_{w} Q_h \phi -\nabla \phi)\cdot\vn
		\rangle_{\partial D}|
		\apprle
		h_D^{-\frac12}\|Q^b_{k,D}u -u\|_{L^2(\partial D)}\
		h_D^{\frac12}\|\nabla_{w} Q_h \phi -\nabla \phi\|_{L^2(\partial D)},
	\end{eqnarray*}
	where
	\begin{align*}
		h_D^{\frac12}&\|\nabla_{w} Q_h \phi -\nabla \phi\|_{L^2(\partial D)}\\
		&\apprle
		h_D^{\frac12}\|\nabla_{w} Q_h \phi -\mathbb{Q}_{k-1,D}\nabla\phi\|_{L^2(\partial D)}
		+
		h_D^{\frac12}\|\mathbb{Q}_{k-1,D}\nabla\phi-\nabla \phi\|_{L^2(\partial D)}\\
		&\apprle
		\|\nabla_{w} Q_h \phi -\mathbb{Q}_{k-1,D}\nabla\phi\|_{L^2(D)}
		+
		h_D|\nabla_{w} Q_h \phi -\mathbb{Q}_{k-1,D}\nabla\phi|_{H^1(D)}\\
		&\quad +
		h_D^{\frac12}\|\mathbb{Q}_{k-1,D}\nabla\phi-\nabla \phi\|_{L^2(\partial D)}\\
		&\apprle
		h_D \|\phi\|_{H^2(D)}
		+
		\|\nabla_{w} Q_h \phi -\mathbb{Q}_{k-1,D}\nabla\phi\|_{L^2(D)} \\
		&\quad +
		h_D^{\frac12}\|\mathbb{Q}_{k-1,D}\nabla\phi-\nabla \phi\|_{L^2(\partial D)}\\
		&\apprle
		h_D \|\phi\|_{H^2(D)}
	\end{align*}
	by Lemma \ref{trace}, Lemma \ref{l4}, Lemma \ref{ler} and Lemma \ref{leq} 
	\begin{eqnarray*}\label{term4}
		|\langle  
		Q^b_{k,D}u -u, 
		(\nabla_{w} Q_h \phi -\nabla \phi)\cdot\vn
		\rangle_{\partial D}|
		\apprle
		h_D^{k+1} \|u\|_{H^{k+1}(D)}\|\phi\|_{H^2(D)}.
	\end{eqnarray*}
	Then 
	\begin{eqnarray}\label{term1}
		| 	(\nabla_{w} Q_h \phi,\nabla_{w} e_h)_\Omega |
		\apprle
		h^{k+1} \|u\|_{H^{k+1}(\Omega)}\|\phi\|_{H^2(\Omega)}.
	\end{eqnarray}
	By \eqref{error1}, \eqref{l2e3}, \eqref{term2}, \eqref{term3} and \eqref{term1}, we have
	\begin{eqnarray*}
		\|Q_k^0u - u_0\|_{L^2(\Omega)}
		\apprle h^{k+1}\|u\|_{H^{k+1}(\Omega)},
	\end{eqnarray*}
	with 
	$$
	\|Q_k^0u - u\|_{L^2(\Omega)}
	\apprle h^{k+1}\|u\|_{H^{k+1}(\Omega)},
	$$
	the error estimate \eqref{l2e} is obtained.
\end{proof}

\section{The weak Galerkin finite element method for elliptic interface problem}\label{interface}
Let $\Omega$ be a bounded domain with $C^{0,1}$ boundary in $\mathbb{R}^n$, $n=2,3$, $\Gamma\subset \Omega$ be the $C^{0,1}$ interface (Fig.\ref{fg2}), $f\in L^2(\Omega)$, the equation is
\begin{equation}\label{interface-eq}
	\begin{cases}
		\ -\nabla\cdot ({\beta} \nabla u)            &= f, \\
		\ [u]|_\Gamma           &= 0, \\
		\ \left(\beta_1\nabla u\cdot\vn -\beta_2\nabla u\cdot\vn\right)|_\Gamma &= g, \\
		\ u|_{\partial \Omega} &= 0,
	\end{cases}
\end{equation}
where $\vn$ is the outward normal vector to $\partial\Omega_1$, $\beta_1, \beta_2$ are two positive constants defined on $\Omega_1$ and $\Omega_2$ respectively. For simplicity, let
$[u]|_{\Gamma} := u_1|_{\Gamma} -u_2|_{\Gamma} = 0.$
\begin{figure}[H]
	\begin{center}
		\begin{tikzpicture}[scale=0.9] 
			%\path[use as bounding box,draw] 	
			\path[use as bounding box] (-0.18,-0.4) rectangle (4.65,3.5);
			\draw[ yshift=0cm](0,0.5) to [closed, curve through = {(1,-.5 )..(2,-0.6)..(3,-.5)} ] (4,0);
			\filldraw[opacity=0.3,gray, yshift=2.5cm](1,-0.7) to [closed, curve through = {(1.5,-0.4 )..(2,-0.6)..(2.5,-0.4)} ] (3,-0.7);
			\draw[black, yshift=2.5cm](1,-0.7) to [closed, curve through = {(1.5,-0.4 )..(2,-0.6)..(2.5,-0.4)} ] (3,-0.7);
			\node [above right] at (3.5,0.7) {$\Omega_2$};
			\node [above] at (2.3,1) {$\Omega_1$};
			\node [ ] at (1.05,1.2) {$\Gamma$};
		\end{tikzpicture}
	\end{center}
	\caption{A  domain $\Omega = \Omega_1\cup\Gamma\cup\Omega_2$ with $C^{0,1}$ interface $\Gamma$}\label{fg2}
\end{figure}
The mesh $\mathcal{T}_h$ contains shape-regular curvilinear polygons or polyhedrons which have edges or faces as parts of $\partial \Omega$ or  $\Gamma$, and there is no element cross $\Gamma$.
Here we use the same weak Galerkin finite element schemes as in Section \ref{wgfem-scheme}. For any $v\in V_h,$ the weak gradient of $v$ is defined on each element $D$ by \eqref{wg2},   respectively. 
For any $u,v\in V_h$, the bilinear form is defined by $b_{h}(u,v)$
\begin{equation}\label{bhri} 
	b_{h}(u,v) = \sum\limits_{D\in\mathcal{T}_h}\int_{D}{\beta}\nabla_{w} u\cdot\nabla_{w} v\ {\rm d} x.
\end{equation}
The stabilization term is:
\begin{equation}\label{stablei} 
	s_{h}(u,v) = \sum\limits_{D\in\mathcal{T}_h} 
	h_D^{-1}
	\langle  
	u_0-u_b, v_0-v_b
	\rangle_{\partial D}.
\end{equation}
A numerical solution for \eqref{interface-eq} can be obtained by seeking $u_h =(u_0,u_b)\in V^0_h$ such that 
\begin{equation}\label{num-interface}
	b_s(u_h,v)  = (f,v_0)_{\Omega} +\langle g, v_b\rangle_{\Gamma}, \ \forall v=(v_0,v_b)\in V_h^0.
\end{equation}
where $b_s(u_h,v):=b_h(u_h,v)+s_h(u_h,v).$
Then we define {  a new weak norm} as
\begin{equation}\label{|||}
	|||v||| = (b_s(v,v))^{\frac12},\ \forall v\in V_h^0. 
\end{equation}
%\begin{lemma}\label{poincare}
%	Suppose the partition $\mathcal{T}_h$ is shape regular. Then we have
%	$$
%	\|v_0\|_{L^2(\Omega)} \apprle |v|_{k-1,w}, \ \forall v=(v_0,v_b)\in V_h^0,
%	$$
%	the hidden constant only depends on $\rho_D$ and $k$.
%\end{lemma}
%\begin{proof} 
%	The key to prove Lemma \ref{poincare} is that on each edge or face $e$, $v|_e$ is unique. 
%	Then with the same method of Lemma 7.1 in \cite{Lin15}, we have the discrete Poincar$\acute{\rm e}$ inequality.
%\end{proof}
Also, we have the existence and uniqueness of the solution of \eqref{interface-eq} as in \cite{Lin16}.
\subsection{Error Analysis}
Let $u \in H_0^1(\Omega)$ and $u|_{\Omega_i}\in H^2(\Omega_i), i=1,2,$ be the solution of \eqref{interface-eq} and $v\in V_h^0$. 
Then, multiply \eqref{interface-eq} by $v_0$ of $v=(v_0,v_b)\in V_h^0$ we have
\begin{equation}\label{wui}
	\sum\limits_{D\in\mathcal{T}_h} (\beta \nabla u,\nabla v_0)_D
	=(f,v_0)_{\Omega}
	+
	\langle g, v_b \rangle_{\Gamma}
	+
	\sum\limits_{D\in\mathcal{T}_h}
	\langle  v_0-v_b , \beta \nabla u\cdot\vn  \rangle_{\partial D},
\end{equation}
where 
$
\sum\limits_{D\in\mathcal{T}_h}
\langle v_b , \beta \nabla u\cdot\vn  \rangle_{\partial D} 
= 
\sum\limits_{e\in\Gamma}
\langle g,v_b   \rangle_{e}.$\\
It follows from \eqref{wg2}, \eqref{QhQ1} and the integration by parts
\begin{align}\label{wI_hi}
	(\beta\nabla_{w} Q_hu,\nabla_{w}v)_D
	&= (\beta\mathbb{Q}_{k-1,D}\nabla u, \nabla_{w} v)_D 
	+ 
	\langle  
	Q^b_{k,D}u -u, 
	\beta\nabla_{w}v \cdot\vn
	\rangle_{\partial D}   \\
	&= (\beta\nabla u,\nabla v_0)_D
	-
	\langle v_0-v_b,\beta(\mathbb{Q}_{k-1,D}\nabla u)\cdot\vn \rangle_{\partial D} \nonumber
	\\
	&\quad + 
	\langle  
	Q^b_{k,D}u -u, 
	\beta\nabla_{w}v \cdot\vn
	\rangle_{\partial D}.\nonumber
\end{align}
Combining \eqref{wui} and \eqref{wI_hi}, adding $s_h({Q}_h u, v)$ to both sides, we have
\begin{align}\label{w-error2i}
	b_s(Q_h u,v)
	&=
	(f,v_0)_{\Omega} 
	+
	\langle g,v_b \rangle_{\Gamma}
	+
	\sum\limits_{D\in\mathcal{T}_h}
	\langle  
	v_0-v_b , 
	\beta(\nabla u - \mathbb{Q}_{k-1,D} \nabla u)\cdot\vn  
	\rangle_{\partial D} 
	\\
	&\quad +
	\sum\limits_{D\in\mathcal{T}_h}
	\langle  
	Q^b_{k,D}u -u, 
	\beta\nabla_{w}v \cdot\vn
	\rangle_{\partial D}
	+s_h({Q}_h u, v). \nonumber 
\end{align}
Subtracting \eqref{num-interface} from \eqref{w-error2i}, we have the error equation
\begin{align}\label{w-error3i}
	b_s(e_h,v)
	&=
	\sum\limits_{D\in\mathcal{T}_h}
	\langle  
	v_0-v_b , 
	\beta(\nabla u - \mathbb{Q}_{k-1,D} \nabla u)\cdot\vn  
	\rangle_{\partial D}  
	\\
	&\quad +
	\sum\limits_{D\in\mathcal{T}_h}
	\langle  
	Q^b_{k,D}u -u, 
	\beta\nabla_{w}v \cdot\vn
	\rangle_{\partial D}
	+s_h(Q_hu,v).  \nonumber 
\end{align}
where
$$
e_h|_D = (e_0,e_b)_D:= (Q_{k,D}^0 u -u_0, Q_{k,D}^b u -u_b)_D= (Q_hu-u_h)|_{D}
$$
which is the error between the weak Galerkin finite element solution $(u_0,u_b)$ and the $L^2$ projection of the exact solution.
Then we define a norm $\|\cdot\|_\beta$ as
$$
\|\vv\|_\beta^2 := \sum\limits_{D\in\mathcal{T}_h} (\beta\vv, \vv)_D, \ \forall \vv\in [L^2(\Omega)]^n,
$$
and suppose $v\in H^1(\Omega)$, $v|_{\Omega_i}\in H^{k+1}(\Omega_i), i=1,2,$ we define
$$
\|v\|_{k+1,\Omega}^2 = \|v\|_{H^{k+1}(\Omega_1)}^2+\|v\|_{H^{k+1}(\Omega_2)}^2.
$$
\begin{theorem}
	Let $u_h \in V_h^0$ be the weak Galerkin finite element solution of the
	problem \eqref{interface-eq}. Assume that the exact solution is so regular that
	$u|_{\Omega_i} \in H^{k+1} (\Omega_i), i=1,2$. Then we have
	\begin{eqnarray}
		\|\nabla u - \nabla_{w} u_h\|_\beta 
		&\apprle& h^k
		\|u\|_{{k+1}, \Omega}
		, \label{we1i} \\
		\|\nabla u - \nabla u_0\|_\beta 
		&\apprle& 
		h^k
		\|u\|_{{k+1}, \Omega}
		, \label{we2i} 
	\end{eqnarray}
	the hidden constants only depend on $\rho_D$ and $k$.
\end{theorem}
\begin{proof}
	Let $v = e_h$ in \eqref{w-error3i},  we have
	\begin{equation}
		|||e_h|||
		=
		\sum\limits_{D\in\mathcal{T}_h}
		\langle  
		e_0-e_b , 
		\beta(\nabla u - \mathbb{Q}_{k-1,D} \nabla u)\cdot\vn  
		\rangle_{\partial D} 
		+
		\sum\limits_{D\in\mathcal{T}_h}
		\langle  
		Q^b_{k,D}u -u, 
		\beta\nabla_{w}e_h \cdot\vn
		\rangle_{\partial D}
		+s_h(Q_hu,e_h).
	\end{equation}
	It then follows from \eqref{QhQ2} and Lemma \ref{lem_h1}  
	\begin{equation}\label{error1i}
		|||e_h|||^2 \apprle  h^k
		\|u\|_{{k+1}, \Omega}
		|||e_h|||.
	\end{equation}
	Based on \eqref{error1i}, firstly, we prove \eqref{we1i}, 
	\begin{eqnarray*}
		\|\nabla u - \nabla_{w} u_h\|_\beta 
		&\leq&
		\|\nabla u - \mathbb{Q}_{k-1} \nabla u \|_\beta 
		+\|\mathbb{Q}_{k-1} \nabla u  - \nabla_{w}Q_h u\|_\beta
		+\|\nabla_{w}Q_h u - \nabla_{w} u_h\|_\beta, 
	\end{eqnarray*}
	with Lemma \ref{ler} and Lemma \ref{leq} and
	$$
	\|\nabla_{w}(Q_h u - u_h)\|_\beta\leq |||e_h|||
	$$
	we have
	\eqref{we1}.
	
	Secondly, with Lemma \ref{K-bound}, we have
	\begin{eqnarray*}
		\sum_{D\in\mathcal{T}_h}
		\|\beta\nabla (Q^0_{k,D} u -u_h|_{D_0})\|^2_{L^2(D)} 
		&=&
		\sum\limits_{D\in\mathcal{T}_h} \|\beta\nabla  e_0\|^2_{L^2(D)}\\
		&\apprle&
		\sum\limits_{D\in\mathcal{T}_h} \|\beta\nabla_{w} e_h\|^2_{L^2(D)} +h_D^{-1}\|e_b-e_0\|^2_{L^2(\partial D)}\\
		&\apprle&
		|||e_h|||^2
	\end{eqnarray*}
	which means
	$$
	\sum_{D\in\mathcal{T}_h}
	\|\beta\nabla (Q^0_{k,D} u -u_h|_{D_0})\|^2_{L^2(D)}  \apprle h^{2k}
	\|u\|_{{k+1}, \Omega}^2.
	$$
	Also with Lemma \ref{ler}
	$$
	\sum_{D\in\mathcal{T}_h}\|\beta\nabla (Q^0_{k,D} u -u)\|^2_{L^2(D)}  \apprle h^{2k}
	\|u\|_{{k+1}, \Omega}^2,
	$$
	so that we have \eqref{we2i}
	$$
	\|\nabla u - \nabla u_0\|_\beta  
	\apprle 
	h^k
	\|u\|_{{k+1}, \Omega}.
	$$
\end{proof}
\begin{theorem}
	Let $u_h \in V_h$ be the weak Galerkin finite element solution of the
	problem \eqref{interface-eq}. Assume that the exact solution is so regular that
	$u|_{\Omega_i} \in H^{k+1} (\Omega_i), i=1,2$. Then we have
	\begin{eqnarray}\label{l2ei}
		\|u - u_0\|_{L^2(\Omega)}
		&\apprle& h^{k+1}\|u\|_{{k+1}, \Omega},
	\end{eqnarray}
	the hidden constant only depends on $\rho_D$ and $k$.
\end{theorem}
\begin{proof}
	We begin with a dual problem seeking $\phi\in H_0^1(\Omega)$ such that
	$
	-\nabla\cdot ({\beta} \nabla \phi)  = e_0,
	$ and $g=0$ as in \eqref{interface-eq}.
	Suppose we have $\|\phi\|_{2,\Omega}\apprle \|e_0\|_{L^2(\Omega)}$.
	
	Then we have
	\begin{equation}\label{l2e1i}
		\|e_0\|_{L^2(\Omega)}^2 = 
		\sum\limits_{D\in\mathcal{T}_h}(\beta\nabla\phi,\nabla e_0)_D 
		-
		\sum\limits_{D\in\mathcal{T}_h}
		\langle 
		\beta\nabla\phi\cdot\vn, 
		e_0-e_b 
		\rangle_{\partial D}.
	\end{equation}
	Let $u=\phi$ and $v=e_h$ in \eqref{wI_hi}, we have
	\begin{align}\label{l2e2i}
		(\beta\nabla \phi,\nabla e_0)_D	
		&=
		(\beta\nabla_{w} Q_h \phi,\nabla_{w} e_h)_D
		+
		\langle e_0-e_b,\beta(\mathbb{Q}_{k-1,D}\nabla \phi)\cdot\vn \rangle_{\partial D} 
		\\
		&\quad -
		\langle  
		Q^b_{k,D}\phi -\phi, 
		\beta\nabla_{w}e_h \cdot\vn
		\rangle_{\partial D}. \nonumber
	\end{align}
	Combining \eqref{l2e1i} and \eqref{l2e2i}, we have
	\begin{align}\label{l2e3i}
		\|e_0\|_{L^2(\Omega)}^2 
		&= 
		(\beta\nabla_{w} Q_h \phi,\nabla_{w} e_h)_\Omega
		+
		\sum\limits_{D\in\mathcal{T}_h}
		\langle 
		\beta(\mathbb{Q}_{k-1,D}\nabla \phi-\nabla\phi)\cdot\vn, 
		e_0-e_b 
		\rangle_{\partial D}
		\\
		&\quad -
		\sum\limits_{D\in\mathcal{T}_h}
		\langle  
		Q^b_{k,D}\phi -\phi, 
		\beta\nabla_{w}e_h \cdot\vn
		\rangle_{\partial D}. \nonumber
	\end{align}
	By Lemma \ref{leq} and Lemma \ref{lem_h1}, we have
	\begin{align}
		&	\left|
		\sum\limits_{D\in\mathcal{T}_h}
		\langle 
		\beta(\mathbb{Q}_{k-1,D}\nabla \phi-\nabla\phi)\cdot\vn, 
		e_0-e_b 
		\rangle_{\partial D}
		\right|
		\apprle	
		h\|\phi\|_{2,\Omega}|||e_h|||,\label{term2i} 
		\\
		&   \left|	
		\sum\limits_{D\in\mathcal{T}_h}
		\langle  
		Q^b_{k,D}\phi -\phi, 
		\beta\nabla_{w}e_h \cdot\vn
		\rangle_{\partial D}
		\right|
		\apprle	
		h\|\phi\|_{2,\Omega}|||e_h|||.\label{term3i}
	\end{align}	
	Then let $v = Q_h \phi$ in \eqref{w-error3i} 
	\begin{eqnarray*}
		(\beta\nabla_{w} Q_h \phi,\nabla_{w} e_h)_\Omega 
		&=&
		\sum\limits_{D\in\mathcal{T}_h}
		\langle  
		Q^0_{k,D}\phi- Q^b_{k,D}\phi , 
		\beta(\nabla u - \mathbb{Q}_{k-1,D} \nabla u)\cdot\vn  
		\rangle_{\partial D} \label{term11i}\\
		&&+
		\sum\limits_{D\in\mathcal{T}_h}
		\langle  
		Q^b_{k,D}u -u, 
		\beta(\nabla_{w} Q_h \phi -\nabla \phi)\cdot\vn
		\rangle_{\partial D} \label{term12i}
		\\
		&&+s_h(Q_hu,Q_h \phi) -s_h(e_h,Q_h \phi), \label{term13i}
	\end{eqnarray*}
	where
	$$
	\sum\limits_{D\in\mathcal{T}_h}
	\langle  
	Q^b_{k,D}u -u, 
	\beta\nabla \phi \cdot\vn
	\rangle_{\partial D} = 0.
	$$
	Same as the proof of Theorem 8.2 in \cite{Lin15}, we have
	$$
	\sum\limits_{D\in\mathcal{T}_h}
	\langle  
	Q^0_{k,D}\phi- Q^b_{k,D}\phi , 
	\beta(\nabla u - \mathbb{Q}_{k-1,D} \nabla u)\cdot\vn  
	\rangle_{\partial D}
	\apprle
	h^{k+1}\|u\|_{k+1,\Omega}\|\phi\|_{2,\Omega}
	$$
	and
	$$
	|s_h(Q_hu,Q_h \phi)| +|s_h(e_h,Q_h \phi)|\apprle h^{k+1}\|u\|_{{k+1},\Omega}\|\phi\|_{2,\Omega}.
	$$	
	Then by Lemma \ref{l4}, Lemma \ref{ler}, Lemma \ref{leq} 
	\begin{eqnarray*}
		|\langle  
		Q^b_{k,D}u -u, 
		\beta(\nabla_{w} Q_h \phi -\nabla \phi)\cdot\vn
		\rangle_{\partial D}|
		&\apprle&
		h_D^{-\frac12}\|Q^b_{k,D}u -u\|_{L^2(\partial D)}\
		\beta h_D^{\frac12}\|\nabla_{w} Q_h \phi -\nabla \phi\|_{L^2(\partial D)},
	\end{eqnarray*}
	where
	\begin{eqnarray*}
		h_D^{\frac12}\|\nabla_{w} Q_h \phi -\nabla \phi\|_{L^2(\partial D)}
		&\apprle&
		h_D^{\frac12}\|\nabla_{w} Q_h \phi -\mathbb{Q}_{k-1,D}\nabla\phi\|_{L^2(\partial D)}
		+
		h_D^{\frac12}\|\mathbb{Q}_{k-1,D}\nabla\phi-\nabla \phi\|_{L^2(\partial D)}\\
		&\apprle&
		\|\nabla_{w} Q_h \phi -\mathbb{Q}_{k-1,D}\nabla\phi\|_{L^2(D)}
		+
		h_D|\nabla_{w} Q_h \phi -\mathbb{Q}_{k-1,D}\nabla\phi|_{H^1(D)}\\
		&&+
		h_D^{\frac12}\|\mathbb{Q}_{k-1,D}\nabla\phi-\nabla \phi\|_{L^2(\partial D)}\\
		&\apprle&
		h_D \|\phi\|_{H^2(D)}
		+
		\|\nabla_{w} Q_h \phi -\mathbb{Q}_{k-1,D}\nabla\phi\|_{L^2(D)} \\
		&&+
		h_D^{\frac12}\|\mathbb{Q}_{k-1,D}\nabla\phi-\nabla \phi\|_{L^2(\partial D)}\\
		&\apprle&
		h_D \|\phi\|_{H^2(D)}
	\end{eqnarray*}
	so that 
	% \begin{eqnarray*}\label{term4i}
		$		|\langle  
		Q^b_{k,D}u -u, 
		\beta(\nabla_{w} Q_h \phi -\nabla \phi)\cdot\vn
		\rangle_{\partial D}|
		\apprle
		h_D^{k+1} \|u\|_{H^{k+1}(D)}\|\phi\|_{H^2(D)}.
		$	
		%\end{eqnarray*}
		\\
		Then we have
		\begin{eqnarray}\label{term1i}
			|(\beta\nabla_{w} Q_h \phi,\nabla_{w} e_h)_\Omega |
			\apprle
			h^{k+1} \|u\|_{{k+1},\Omega}\|\phi\|_{2,\Omega}.
		\end{eqnarray}
		By \eqref{error1i}, \eqref{l2e3i}, \eqref{term2i}, \eqref{term3i} and \eqref{term1i}, we have
		\begin{eqnarray*}
			\|Q_k^0u - u_0\|_{L^2(\Omega)}
			\apprle h^{k+1}\|u\|_{{k+1},\Omega},
		\end{eqnarray*}
		with 
		$$
		\|Q_k^0u - u\|_{L^2(\Omega)}
		\apprle h^{k+1}\|u\|_{{k+1},\Omega},
		$$
		the error estimate \eqref{l2ei} is obtained.
	\end{proof}

	\section{Numerical Experiments}\label{tests}
	In this section, we focus on 2D problems with curved boundaries or interfaces. The main difference compared with the classical weak Galerkin method is the construction of basis functions on curved sides. For the following tests, each element in the mesh has four sides, and one or more sides are curves. We choose two types of curved sides, the special Lipschitz continuous ones formed by connecting points with short lines, and the smooth ones. Basis functions are linearly independent traces of $\mathbb{P}_k(\mathbb{R}^2)$ on the curved side, for example, it's easy to prove if the side is not straight, $P_1$ basis functions are always {  chosen}: $1, x, y$. Let $x^0=y^0=1$, and the set of traces for $\mathbb{P}_k(\mathbb{R}^2), k\geq 2$ on the side be 
	$$T(\mathbb{P}_k):=\{x^iy^j| 0\leq i,j {\rm \ and \ } i+j<=k\}.$$ 
	Instead of determining if the traces are linearly independent analytically, in practice, we can let $(x_m,y_m), 1\leq m\leq  M$ be sample points on the curved side, the number of those points should be sufficiently large.  For each $v(x,y) \in T(\mathbb{P}_k)$, we can get a vector 
	$$(v(x_1,y_1), v(x_2,y_2),\cdots, v(x_M,y_M)),$$ then collect all those vectors and choose the linearly independent ones, the corresponding traces are basis functions on the curved side. Numerical integration on the curved side or curvilinear element is simple, we just need to make sure it is accurate enough to calculate convergence rates. The algorithm in 2D can be easily generalized to 3D problems.
	\subsection{Poisson's Equation}
	This subsection is to verify that elements with Lipschitz continuous or smooth curved edges can also guarantee optimal convergence rates for solving Poisson's equation as long as they satisfy the shape-regular assumptions.
	Let 
	$$
	u(x,y) = \sin(\pi x)\sin(\pi y)
	$$
	be the exact solution to Poisson's equation on a given domain with Dirichlet/Neumann Boundary conditions, we use $P_1$ and $P_2$ elements in Example \ref{ex:SE} to Example \ref{ex:DNB-C}.
	\begin{ex}\label{ex:SE}
		Let the domain be $[0,1]\times [0,1]$, see Figure \ref{fig:SE}. We consider Poisson's equation with the Dirichlet boundary condition (on the blue lines). In the mesh, each element has a  Lipschitz continuous curved side which contains three short lines, basis functions are defined on the whole green-colored curved side, there are three for $P_1$ element and six for $P_2$ element. From Table \ref{tab:SE}, we can see the convergence rates for $L^2$ and $H^1$ errors are optimal.
	\end{ex}
	\begin{figure}[ht!]
		\centering
		\includegraphics[width=0.26\textwidth]{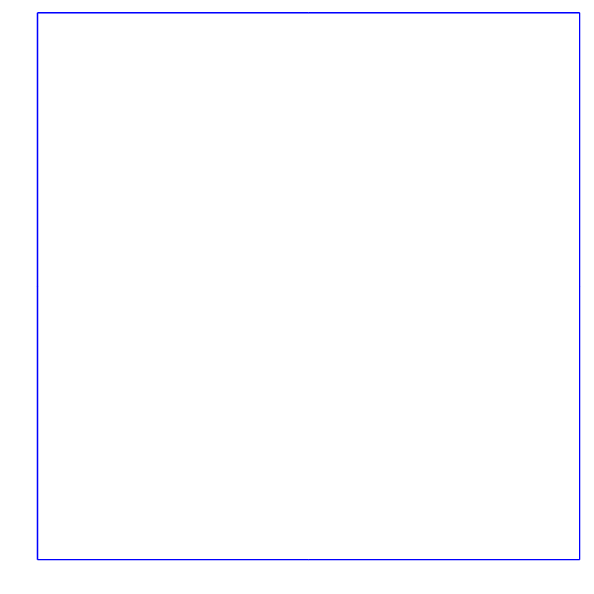}
		\includegraphics[width=0.26\textwidth]{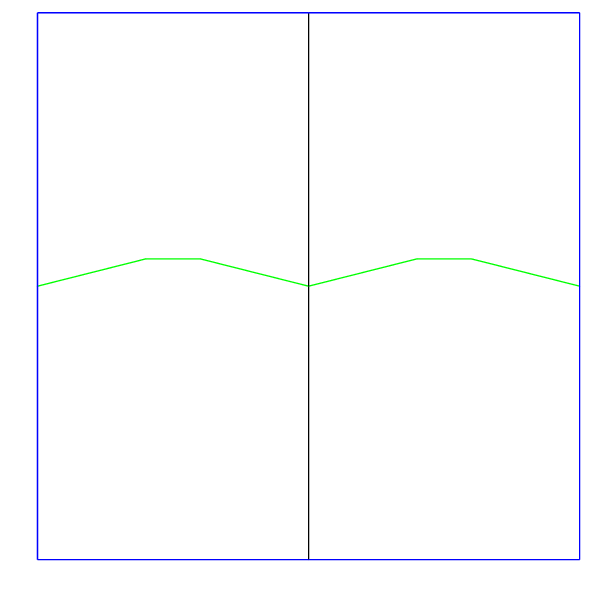}
		\includegraphics[width=0.26\textwidth]{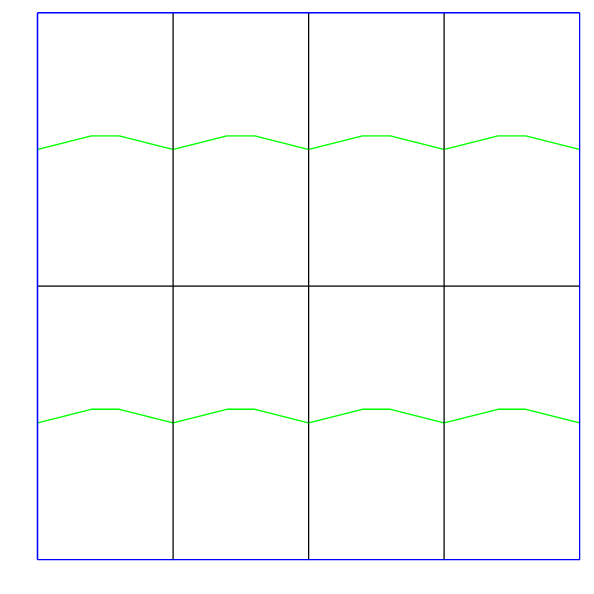}
		\caption{Example \ref{ex:SE}. Left: domain; Middle: initial mesh; Right: refined mesh}
		\label{fig:SE}
	\end{figure}
	\begin {table}[ht!]
	\caption {Example \ref{ex:SE}. Convergence Rates for $P_1$ and $P_2$ Elements, ${\rm d} \approx \sqrt{2}$}  
	\begin{center}
		\begin{tabular}{|c|c|c|c|c|c|c|c|c|} 
			\hline	
			h$\approx$ & $L^2$ error & $P_1$ & $H^1$ error & $P_1$ & $L^2$ error & $P_2$ & $H^1$ error & $P_2$ \\ \hline 
			{\rm d}/2 & 4.826e-01 & Rate & 1.758e+00 & Rate & 1.516e-01 & Rate & 1.256e+00 & Rate \\ \hline 
			{\rm d}/4 & 1.396e-01 & 1.79 & 7.943e-01 & 1.15 & 1.825e-02 & 3.05 & 3.017e-01 & 2.06 \\ \hline 
			{\rm d}/8 & 3.630e-02 & 1.94 & 3.711e-01 & 1.10 & 2.112e-03 & 3.11 & 7.155e-02 & 2.08 \\ \hline 
			{\rm d}/16 & 9.167e-03 & 1.99 & 1.815e-01 & 1.03 & 2.558e-04 & 3.05 & 1.754e-02 & 2.03 \\ \hline 
			{\rm d}/32 & 2.298e-03 & 2.00 & 9.019e-02 & 1.01 & 3.169e-05 & 3.01 & 4.360e-03 & 2.01 \\ \hline 
			{\rm d}/64 & 5.748e-04 & 2.00 & 4.503e-02 & 1.00 & 3.952e-06 & 3.00 & 1.089e-03 & 2.00 \\ \hline 
		\end{tabular}
  \label{tab:SE}
	\end{center}
\end{table}	

%%%%%%%%%%%%%%%%%%%% new example %%%%%%%%%%%%%%%%%%%%%%%-------------------------------------------
{ 
	\begin{ex}\label{ex:SE-n1}
		Following Example \ref{ex:SE}, we define a parameter $M$ which makes the curved side more flat if it is increasing, see Figure \ref{fig:SE} (the middle one with $M=10$) and Figure \ref{fig:SE-n1}. For the same Poisson's equation and $P_1$ element, we test the condition numbers of the stiff matrices on coarse meshes and obtain the convergence rates with fixed $M$ and refined meshes. We observe that the condition number of the stiff matrix is increasing as $M$ increases, see Table \ref{tabSE-c1}.
  However, from Table \ref{tabSE-n1}, we can see that the convergence rates for $L^2$ and $H^1$ errors are still optimal (even slightly better compared with Table \ref{tab:SE} with $P_1$ element). So large condition numbers do not necessarily reduce the accuracy.  
	\end{ex}
 }
	\begin{figure}[ht!]
		\centering
		\includegraphics[width=0.4\textwidth]{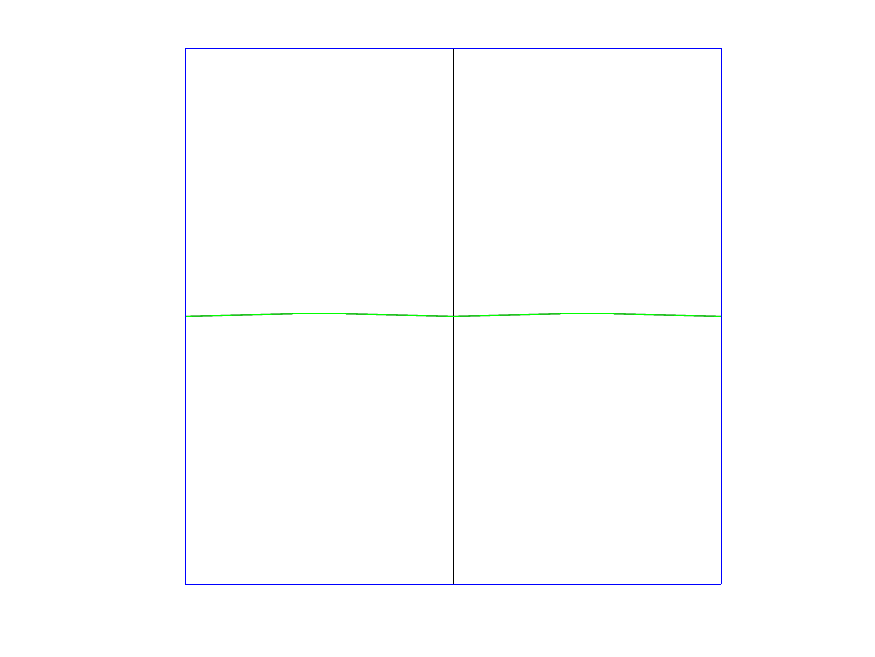}
		\includegraphics[width=0.4\textwidth]{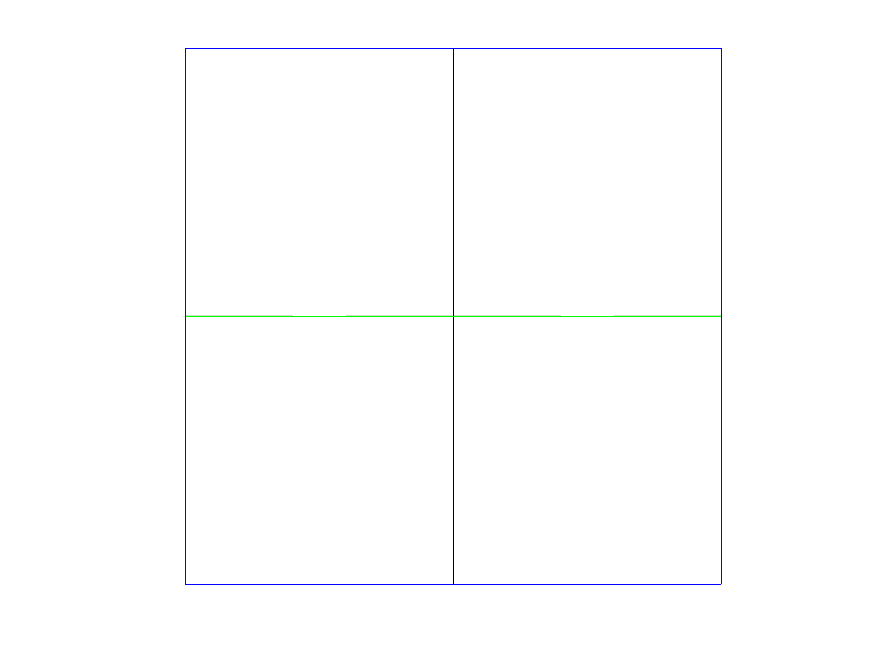}
		\caption{Example \ref{ex:SE-n1}. Meshes withes 4 elements, left:  $M=10^2$;  right: $M=10^4$}
		\label{fig:SE-n1}
	\end{figure}
 
\begin {table}[ht!]
\caption {Example \ref{ex:SE-n1}. Condition Numbers for stiff matrices on coarse meshes with $P_1$ basis functions }  
\begin{center}
\begin{tabular}{|c|c|c|c|c|c|} 
\hline
h$\approx$ & $M=5$ & $M=10$ & $M=10^2$ & $M=10^3$ & $M=10^4$  \\ \hline 
$\sqrt{2}/2$ & 2.5861e+03 & 9.9975e+03 & 9.8530e+05 & 9.8446e+07 & 9.8438e+09  \\ \hline 
$\sqrt{2}/4$ & 1.0442e+04 & 3.7529e+04 & 3.6133e+06 & 3.6097e+08 & 3.6094e+10 \\ \hline 
$\sqrt{2}/8$ & 1.3266e+05 & 1.4676e+05 & 1.3561e+07 & 1.3547e+09 & 1.3547e+11 \\ \hline 
\end{tabular}
\label{tabSE-c1}
\end{center}
\end{table}	

\begin {table}[ht!]
\caption {Example \ref{ex:SE-n1}. Errors for $P_1$ Element with $M=10^2$ and $M=10^4$, ${\rm d} \approx \sqrt{2}$ }  
\begin{center}
\begin{tabular}{|c|c|c|c|c|c|c|c|c|} 
\hline
h$\approx$ & $L^2$ error & $10^2$ & $H^1$ error & $10^2$ & $L^2$ error & $10^4$ & $H^1$ error & $10^4$ \\ \hline 
{\rm d}/2 & 4.785e-01 & Rate & 1.760e+00 & Rate & 4.785e-01 & Rate & 1.760e+00 & Rate \\ \hline 
{\rm d}/4 & 1.378e-01 & 1.80 & 7.868e-01 & 1.16 & 1.378e-01 & 1.80 & 7.867e-01 & 1.16 \\ \hline 
{\rm d}/8 & 3.572e-02 & 1.95 & 3.674e-01 & 1.10 & 3.572e-02 & 1.95 & 3.673e-01 & 1.10 \\ \hline 
{\rm d}/16 & 9.012e-03 & 1.99 & 1.796e-01 & 1.03 & 9.010e-03 & 1.99 & 1.795e-01 & 1.03 \\ \hline 
{\rm d}/32 & 2.258e-03 & 2.00 & 8.923e-02 & 1.01 & 2.258e-03 & 2.00 & 8.922e-02 & 1.01 \\ \hline 
{\rm d}/64 & 5.648e-04 & 2.00 & 4.455e-02 & 1.00 & 5.647e-04 & 2.00 & 4.454e-02 & 1.00 \\ \hline 
\end{tabular}
\label{tabSE-n1}
\end{center}
\end{table}	
%%%%%%%%%%%%%%%%%%% end %%%%%%%%%%%%%%%%%%%%%%%%%%%%%%%%-------------------------------------------

\begin{ex}\label{ex:SE-C}
	Let the domain be $[0,1]\times [0,1]$, see Figure \ref{fig:SE-C}. Still, we consider Poisson's equation with the Dirichlet boundary condition (on the blue lines). In the mesh, each element has a smooth curved side which is part of a quadratic function, basis functions are defined on the green-colored smooth curved side, there are three for $P_1$ element, and five for $P_2$ element. From Table \ref{tab:SE-C}, we can see the convergence rates for $L^2$ and $H^1$ errors are optimal.
\end{ex}
\begin{figure}[ht!]
	\centering
	\includegraphics[width=0.26\textwidth]{./fig/SE-DB-1.eps}
	\includegraphics[width=0.26\textwidth]{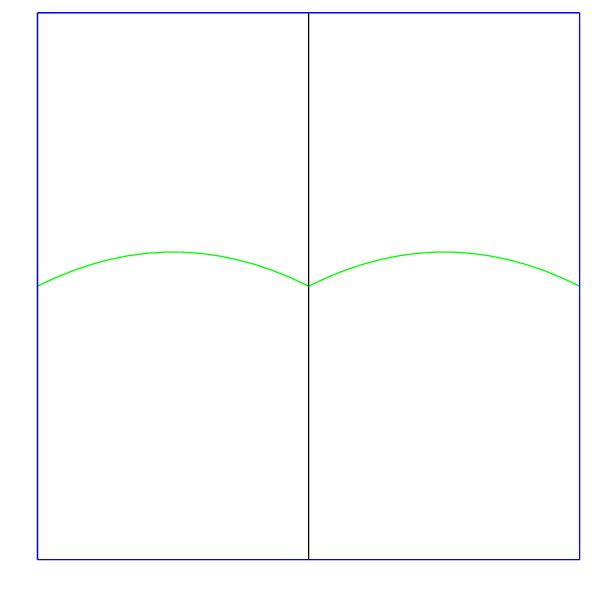}
	\includegraphics[width=0.26\textwidth]{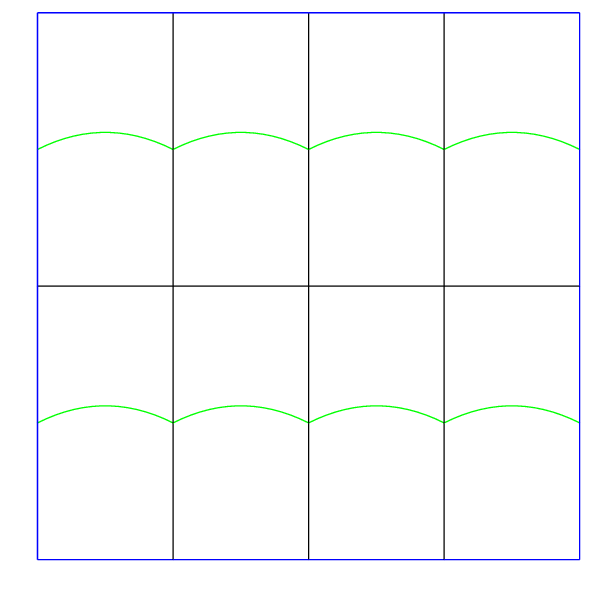}
	\caption{Example \ref{ex:SE-C}. Left: domain; Middle: initial mesh; Right: refined mesh}
	\label{fig:SE-C}
\end{figure}
\begin {table}[ht]
\caption {Example \ref{ex:SE-C}. Convergence Rates for $P_1$ and $P_2$ Elements, ${\rm d} \approx \sqrt{2}$} 
\begin{center}
	\begin{tabular}{|c|c|c|c|c|c|c|c|c|} 
		\hline	
		h$\approx$ & $L^2$ error & $P_1$ & $H^1$ error & $P_1$ & $L^2$ error & $P_2$ & $H^1$ error & $P_2$ \\ \hline 
		{\rm d}/2 & 4.871e-01 & Rate & 1.757 & Rate & 1.526e-01 & Rate & 1.266 & Rate \\ \hline 
		{\rm d}/4 & 1.414e-01 & 1.78 & 8.003e-01 & 1.13 & 1.864e-02 & 3.03 & 3.082e-01 & 2.04 \\ \hline 
		{\rm d}/8 & 3.683e-02 & 1.94 & 3.743e-01 & 1.10 & 2.183e-03 & 3.09 & 7.387e-02 & 2.06 \\ \hline 
		{\rm d}/16 & 9.309e-03 & 1.98 & 1.831e-01 & 1.03 & 2.661e-04 & 3.04 & 1.819e-02 & 2.02 \\ \hline 
		{\rm d}/32 & 2.334e-03 & 2.00 & 9.100e-02 & 1.01 & 3.303e-05 & 3.01 & 4.531e-03 & 2.01 \\ \hline 
		{\rm d}/64 & 5.839e-04 & 2.00 & 4.543e-02 & 1.00 & 4.121e-06 & 3.00 & 1.131e-03 & 2.00 \\ \hline 
	\end{tabular}
 \label{tab:SE-C}
\end{center}
\end{table}	
\begin{ex}\label{ex:DNB} Domain with Lipschitz continuous boundaries,
let $N=192$ and
$$x_i=r\sin(\theta_i),\ y_i=r\cos(\theta_i);\ \theta_i=\frac{\pi}{N},\ i=0,1,\cdots,N$$
if $r=1$, then connect points $(x_i,y_i)$ and $(x_{i+1},y_{i+1})$ by straight lines, we obtain the curve $C_0$ (blue color); similarly, for $r=1.2$, we have curve $C_1$ (red color), the domain is bounded by $C_0$ and $C_1$, points are marked on curves, see Figure \ref{fig:DNB}. We consider Poisson's equation with the Dirichlet boundary condition on the blue curve and the Neumann Boundary condition on the red curve. In Figure \ref{fig:DNB}, each element has at least one Lipschitz continuous curved side which contains multiple short lines, basis functions are defined on the whole curved side. From Table \ref{tab:DNB}, we can see the convergence rates for $L^2$ and $H^1$ errors are optimal.
\end{ex}
\begin{figure}[ht!]
\centering
\includegraphics[width=0.2\textwidth]{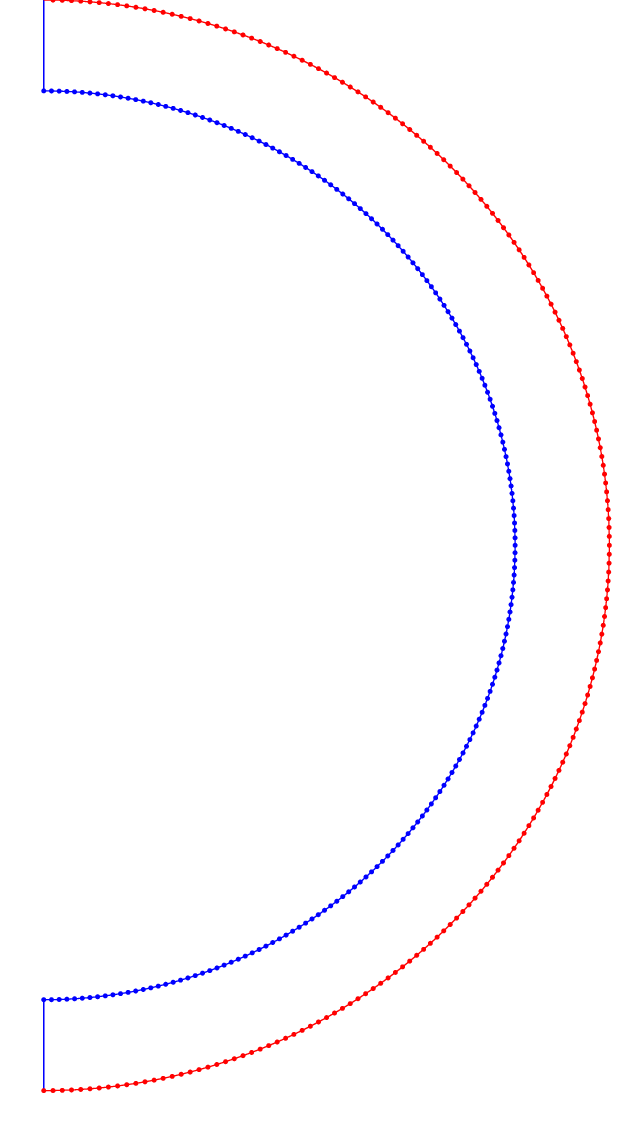}
\hspace{1cm}
\includegraphics[width=0.2\textwidth]{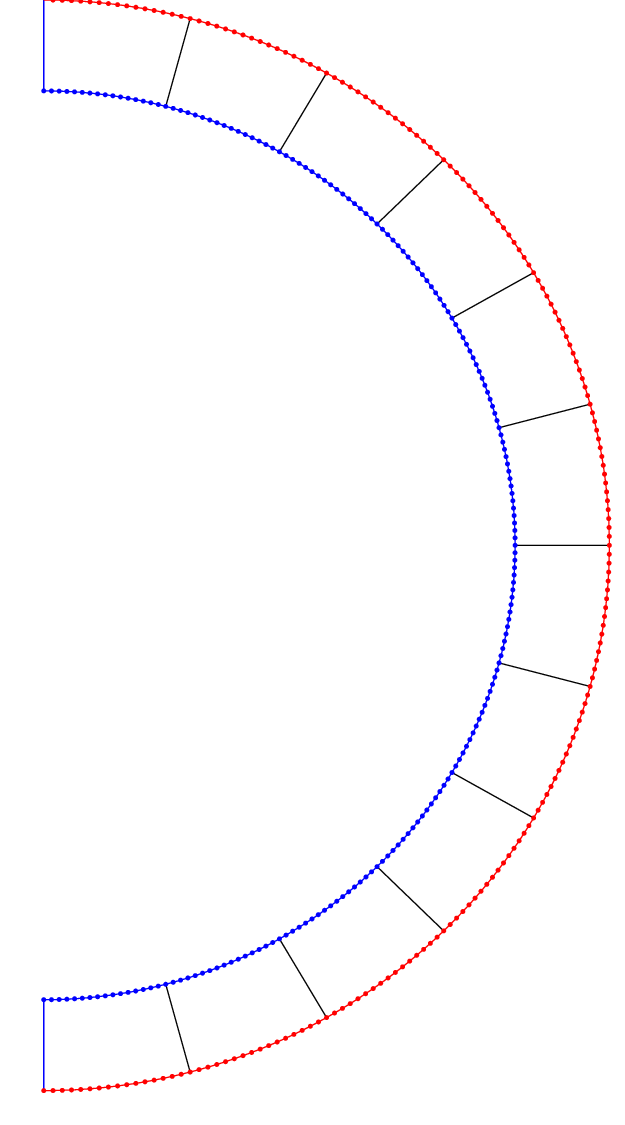}
\hspace{1cm}
\includegraphics[width=0.2\textwidth]{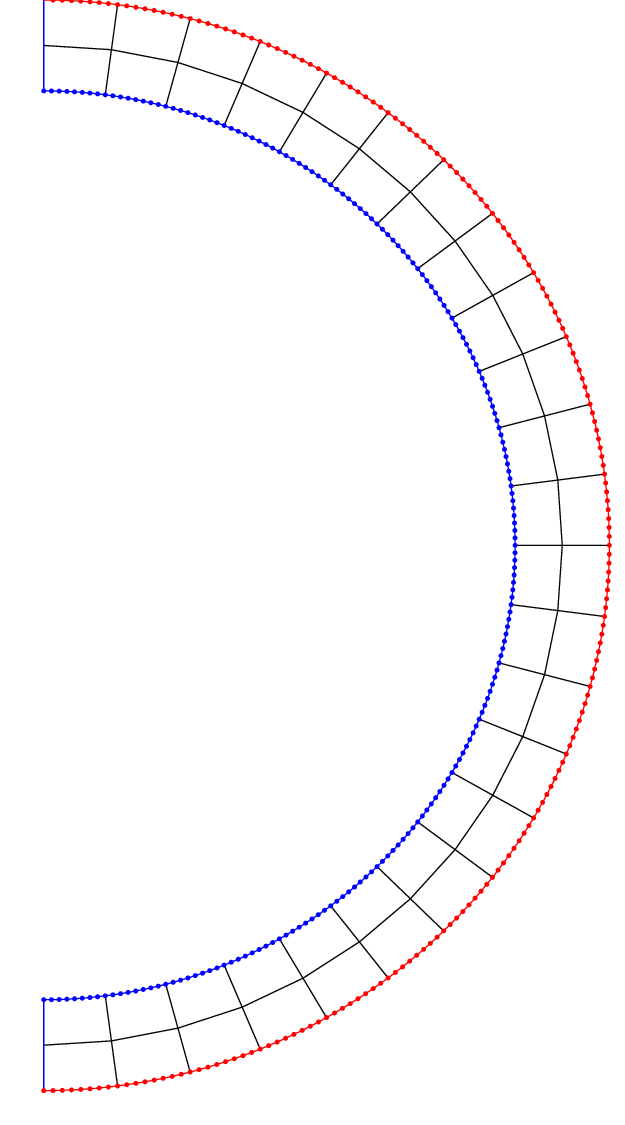}
\caption{Example \ref{ex:DNB}. Left: domain; Middle: initial mesh; Right: refined mesh}
\label{fig:DNB}
\end{figure}
\begin {table}[ht!]
\caption {Example \ref{ex:DNB}. Convergence Rates for $P_1$ and $P_2$ Elements, ${\rm d} \approx 0.70$}  
\begin{center}
\begin{tabular}{|c|c|c|c|c|c|c|c|c|} 
	\hline	
	h$\approx$ & $L^2$ error & $P_1$ & $H^1$ error & $P_1$ & $L^2$ error & $P_2$ & $H^1$ error & $P_2$ \\ \hline 
	{\rm d}/2 & 5.407e-02 & Rate & 7.279e-01 & Rate & 2.263e-02 & Rate & 4.034e-01 & Rate \\ \hline 
	{\rm d}/4 & 1.649e-02 & 1.71 & 3.172e-01 & 1.20 & 3.168e-03 & 2.84 & 1.044e-01 & 1.95 \\ \hline 
	{\rm d}/8 & 4.213e-03 & 1.97 & 1.465e-01 & 1.11 & 4.368e-04 & 2.86 & 2.740e-02 & 1.93 \\ \hline 
	{\rm d}/16 & 1.051e-03 & 2.00 & 7.038e-02 & 1.06 & 5.737e-05 & 2.93 & 7.049e-03 & 1.96 \\ \hline 
	{\rm d}/32 & 2.624e-04 & 2.00 & 3.468e-02 & 1.02 & 7.302e-06 & 2.97 & 1.783e-03 & 1.98 \\ \hline 
\end{tabular}
\label{tab:DNB}
\end{center}
\end{table}	
\begin{ex}\label{ex:DNB-C}
Domain with smooth curved boundaries:
\begin{align*}
& \Omega:=\{(x,y)|x=r\sin(\theta),y=r\cos(\theta),1\le r \le 1.2,\ 0\leq \theta\leq \pi \}. 
\end{align*}
On this domain, we consider Poisson’s equation with the Dirichlet boundary condition
on the blue curve and the Neumann Boundary condition on the red curve, see Figure \ref{fig:DNB-C}. In Figure \ref{fig:DNB-C}, each
element has at least one smooth curved side, basis functions are defined on the curved side. From Table \ref{tab:DNB-C}, we can see the convergence
rates for $L^2$ and $H^1$ errors are optimal.

\end{ex}
\begin{figure}[ht!]
\centering
\includegraphics[width=0.2\textwidth]{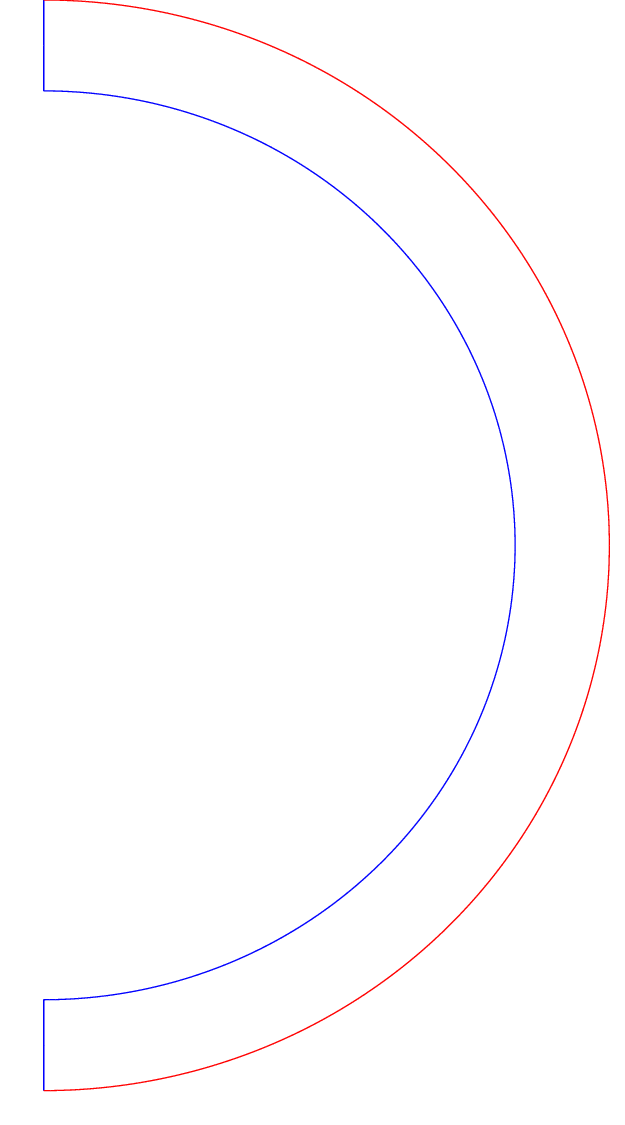}
\hspace{1cm}
\includegraphics[width=0.2\textwidth]{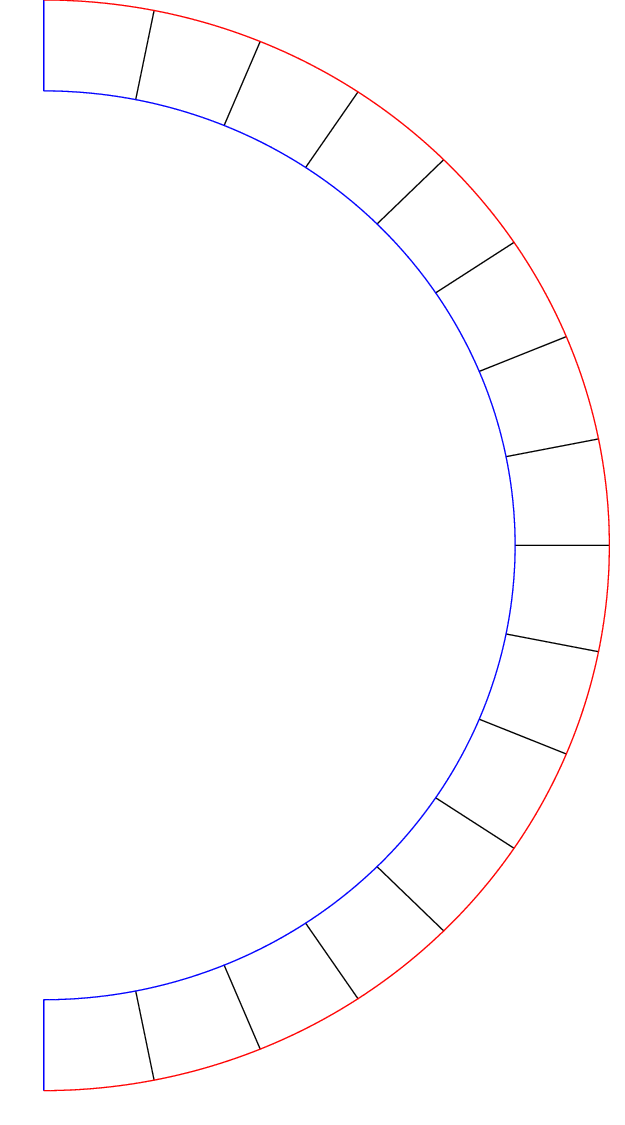}
\hspace{1cm}
\includegraphics[width=0.2\textwidth]{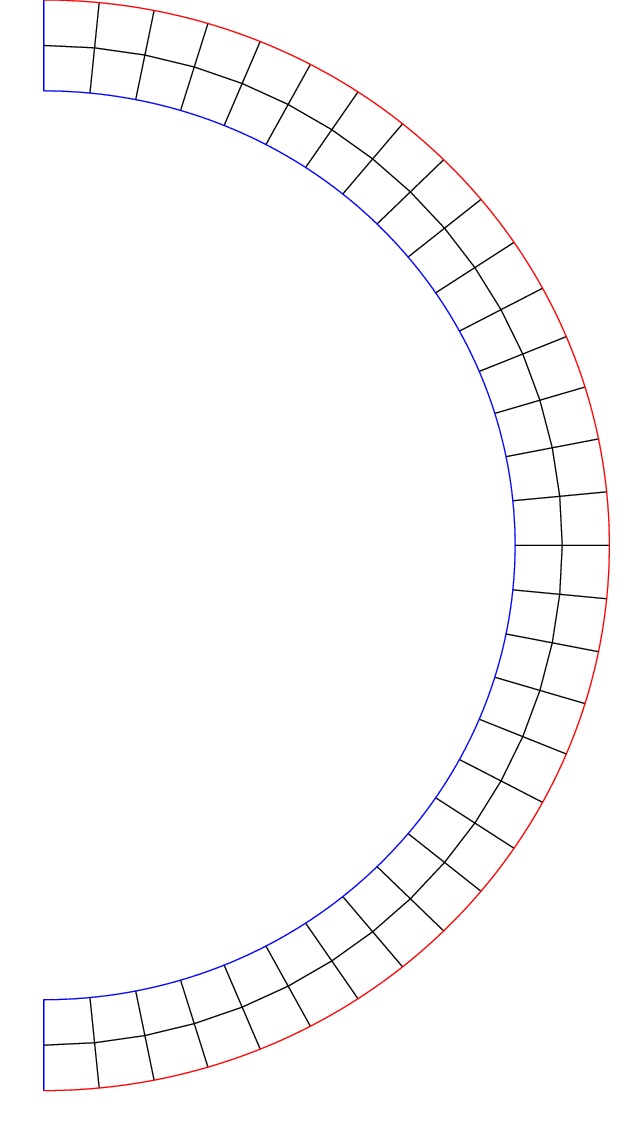}
\caption{Example \ref{ex:DNB-C}. Left: domain; Middle: initial mesh; Right: refined mesh}
\label{fig:DNB-C}
\end{figure}
\begin {table}[ht!]
\caption {Example \ref{ex:DNB-C}. Convergence Rates for $P_1$ and $P_2$ Elements, ${\rm d} \approx 0.58$}  
\begin{center}
\begin{tabular}{|c|c|c|c|c|c|c|c|c|} 
\hline
h$\approx$ & $L^2$ error & $P_1$ & $H^1$ error & $P_1$ & $L^2$ error & $P_2$ & $H^1$ error & $P_2$ \\ \hline 
{\rm d}/2 & 3.548e-02 & Rate & 5.707e-01 & Rate & 1.582e-02 & Rate & 3.052e-01 & Rate \\ \hline 
{\rm d}/4 & 1.158e-02 & 1.62 & 2.507e-01 & 1.19 & 2.086e-03 & 2.92 & 7.838e-02 & 1.96 \\ \hline 
{\rm d}/8 & 3.029e-03 & 1.93 & 1.167e-01 & 1.10 & 2.722e-04 & 2.94 & 2.020e-02 & 1.96 \\ \hline 
{\rm d}/16 & 7.645e-04 & 1.99 & 5.656e-02 & 1.04 & 3.474e-05 & 2.97 & 5.130e-03 & 1.98 \\ \hline 
{\rm d}/32 & 1.915e-04 & 2.00 & 2.800e-02 & 1.01 & 4.377e-06 & 2.99 & 1.291e-03 & 1.99 \\ \hline 
\end{tabular}
\label{tab:DNB-C}
\end{center}
\end{table}

\subsection{Elliptic Interface Problem}
Let $\beta_1=\beta_2=1$ and the exact solution be
$$
u = 
\left\{
\begin{aligned} 
&\sin(\pi x)\sin(\pi y) - 1, & {\rm \ on\ } \Omega_1 \\ 
& \sin(\pi x)\sin(\pi y),& {\rm \ on\ } \Omega_2
\end{aligned}
\right.
$$
for \eqref{interface-eq} with non-homogeneous jump condition.  Suppose $r_0=1, r_1 =1.12,  r_2 = 1.24, r_3=1.36$, then we use $P_1$, $P_2$ elements in examples \ref{ex:CIB} and \ref{ex:CIB-C}. Except for the new basis functions, the employed numerical scheme is the same as  in \cite{Lin16}.
\begin{ex}\label{ex:CIB} Domain with Lipschitz continuous boundaries/interfaces,
let $N=192$ and
$$x_i=r\sin(\theta_i),\ y_i=r\cos(\theta_i);\ \theta_i=\frac{\pi}{N},\ i=0,1,\cdots,N$$
if $r=r_0$, then connect points $(x_i,y_i)$ and $(x_{i+1},y_{i+1})$ by straight lines, we obtain the curve $C_0$; similarly, for $r=r_1, r_2, r_3$, we have curves $C_1,C_2,C_3$. The points are marked on curves in Figure \ref{fig:CIB}. $\Omega_1$ is the region between $C_1$ and $C_2$, the two red cures in the left domain of Figure \ref{fig:CIB}; $\Omega_2$ is the whole region with boundaries $C_0$ and $C_3$, excluding $\Omega_1$. We consider the elliptic interface problem with the Dirichlet boundary condition on the blue curve. In Figure \ref{fig:CIB}, each element has at least one Lipschitz continuous curved side which contains multiple short lines. Basis functions are defined on the whole curved side. From Table \ref{tab:CIB}, we can see the convergence rates for $L^2$ and $H^1$ errors are optimal.
\end{ex}
\begin{figure}[ht!]
\centering
\includegraphics[width=0.2\textwidth]{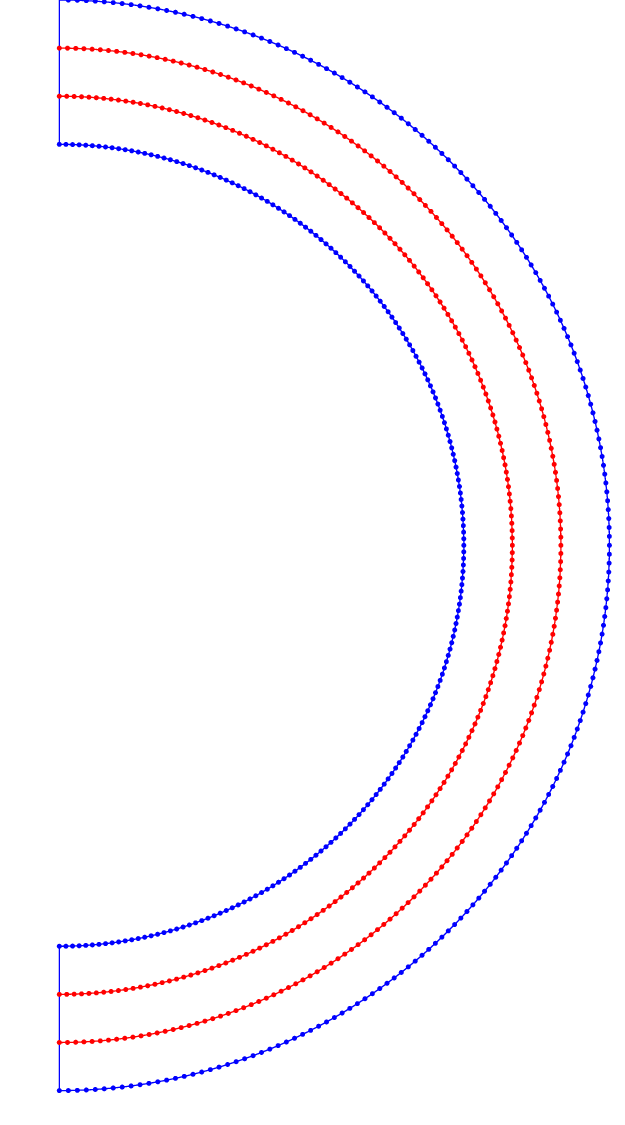}
\hspace{1cm}
\includegraphics[width=0.2\textwidth]{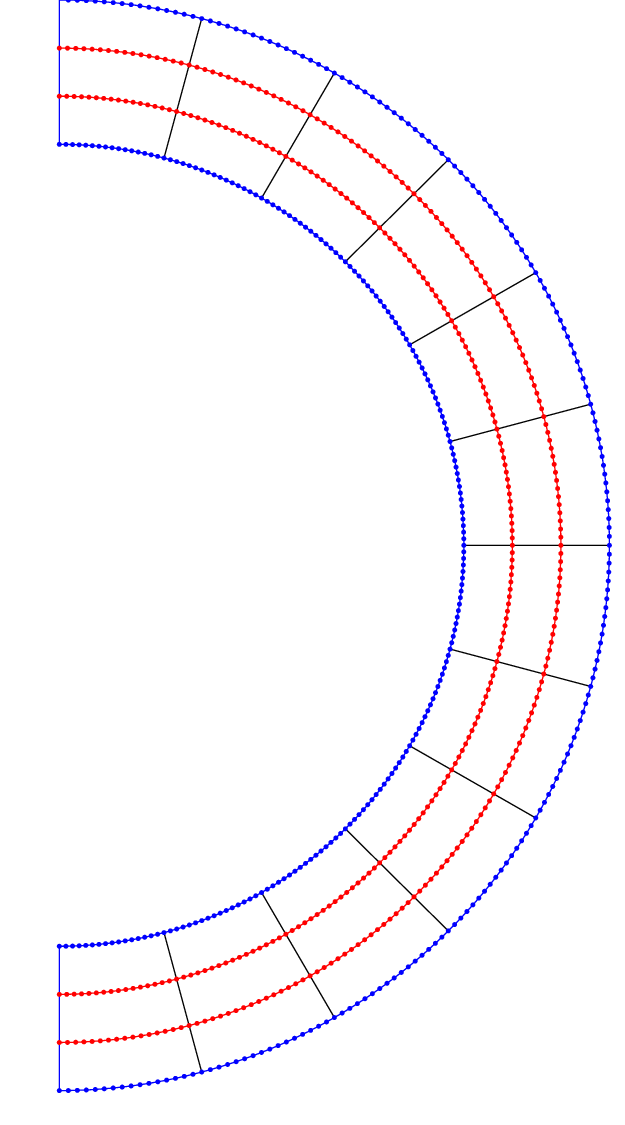}
\hspace{1cm}
\includegraphics[width=0.2\textwidth]{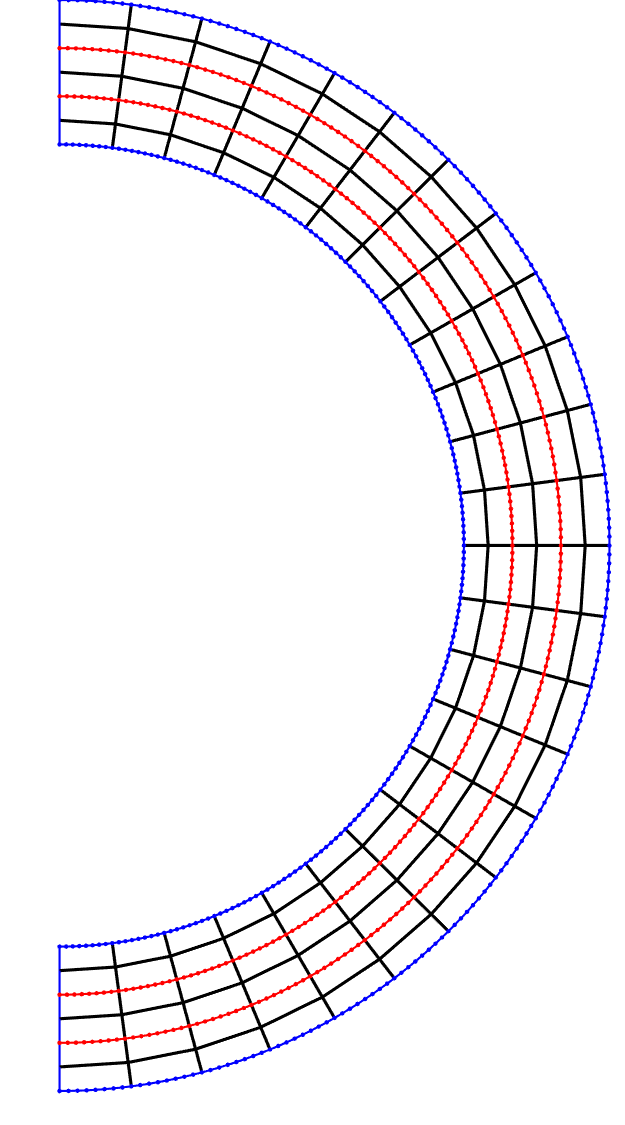}
\caption{Example \ref{ex:CIB}. Left: domain; Middle: initial mesh; Right: refined mesh}
\label{fig:CIB}
\end{figure}
\begin {table}[ht!]
\caption {Example \ref{ex:CIB}. Convergence Rates for $P_1$ and $P_2$ Elements, ${\rm d}\approx 0.72$}  
\begin{center}
\begin{tabular}{|c|c|c|c|c|c|c|c|c|} 
\hline	
h$\approx$ & $L^2$ error & $P_1$ & $H^1$ error & $P_1$ & $L^2$ error & $P_2$ & $H^1$ error & $P_2$ \\ \hline 
{\rm d}/2 & 1.055e-01 & Rate & 8.156e-01 & Rate & 2.488e-02 & Rate & 4.146e-01 & Rate \\ \hline 
{\rm d}/4 & 2.739e-02 & 1.95 & 4.027e-01 & 1.02 & 3.845e-03 & 2.69 & 1.141e-01 & 1.86 \\ \hline 
{\rm d}/8 & 6.755e-03 & 2.02 & 1.955e-01 & 1.04 & 5.574e-04 & 2.79 & 3.108e-02 & 1.88 \\ \hline 
{\rm d}/16 & 1.675e-03 & 2.01 & 9.587e-02 & 1.03 & 7.509e-05 & 2.89 & 8.138e-03 & 1.93 \\ \hline 
{\rm d}/32 & 4.177e-04 & 2.00 & 4.763e-02 & 1.01 & 9.727e-06 & 2.95 & 2.080e-03 & 1.97 \\ \hline 
\end{tabular}
\label{tab:CIB}
\end{center}
\end{table}		
\begin{ex}\label{ex:CIB-C}
Domain with smooth curved boundaries and interfaces:
\begin{align*}
& \Omega_1:=\{(x,y)|x=r\sin(\theta),y=r\cos(\theta),r_1\le r \le r_2,\ 0\leq \theta\leq \pi \} \\
&\Omega_2:=\{(x,y)|x=r\sin(\theta),y=r\cos(\theta),r\in[r_0,r_1]\cup [r_2,r_3],\ 0\leq \theta\leq \pi \}.
\end{align*}
We consider the elliptic interface problem with the Dirichlet boundary condition
on the blue curve, see Figure \ref{fig:CIB-C}. In Figure \ref{fig:CIB-C}, each
element has at least one smooth curved side.
Basis functions are defined on the curved side. From Table \ref{tab:CIB-C}, we can see the convergence
rates for $L^2$ and $H^1$ errors are optimal.
\end{ex}
\begin{figure}[ht!]
\centering
\includegraphics[width=0.2\textwidth]{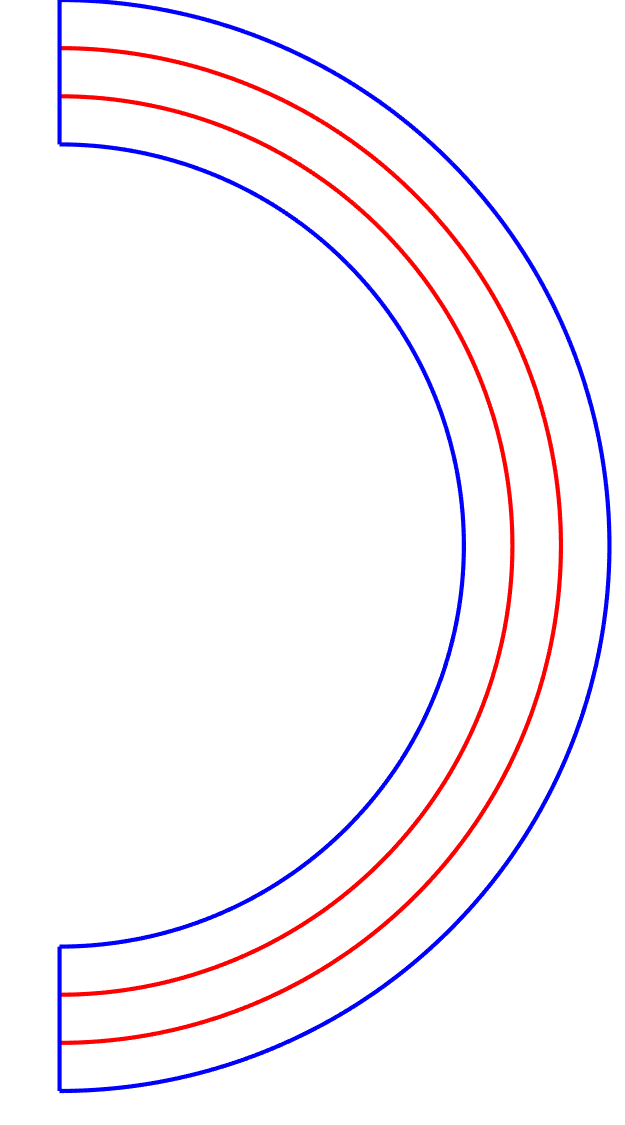}
\hspace{1cm}
\includegraphics[width=0.2\textwidth]{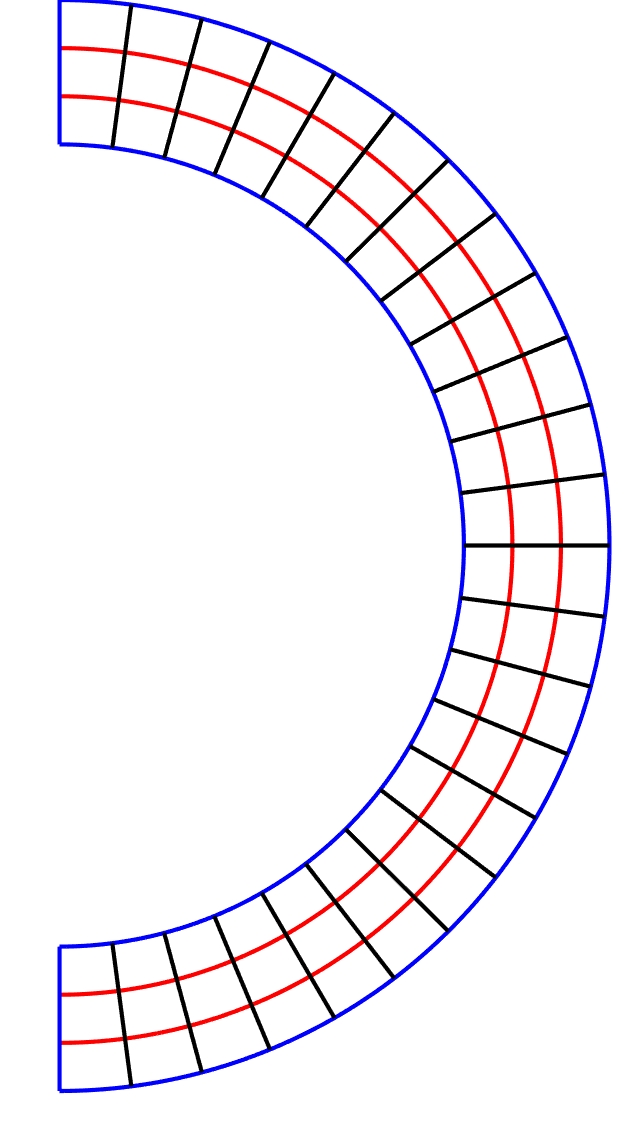}
\hspace{1cm}
\includegraphics[width=0.2\textwidth]{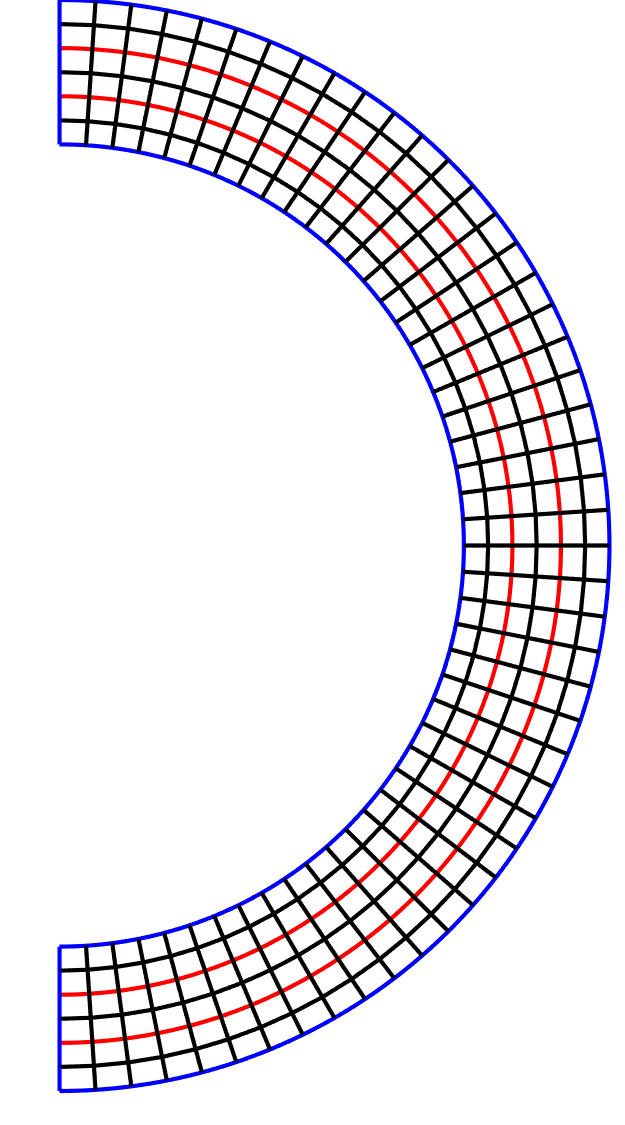}
\caption{Example \ref{ex:CIB-C}. Left: domain; Middle: initial mesh; Right: refined mesh}
\label{fig:CIB-C}
\end{figure}
\begin {table}[ht!]
\caption {Example \ref{ex:CIB-C}. Convergence Rates for $P_1$ and $P_2$ Elements, ${\rm d}\approx 0.42$}  
\begin{center}
\begin{tabular}{|c|c|c|c|c|c|c|c|c|} 
\hline	
h$\approx$ & $L^2$ error & $P_1$ & $H^1$ error & $P_1$ & $L^2$ error & $P_2$ & $H^1$ error & $P_2$ \\ \hline 
{\rm d}/2 & 5.072e-02 & Rate & 4.716e-01 & Rate & 6.251e-03 & Rate & 1.773e-01 & Rate \\ \hline 
{\rm d}/4 & 1.297e-02 & 1.97 & 2.251e-01 & 1.07 & 8.378e-04 & 2.90 & 4.620e-02 & 1.94 \\ \hline 
{\rm d}/8 & 3.248e-03 & 2.00 & 1.094e-01 & 1.04 & 1.093e-04 & 2.94 & 1.187e-02 & 1.96 \\ \hline 
{\rm d}/16 & 8.119e-04 & 2.00 & 5.416e-02 & 1.01 & 1.395e-05 & 2.97 & 3.008e-03 & 1.98 \\ \hline 
{\rm d}/32 & 2.030e-04 & 2.00 & 2.700e-02 & 1.00 & 1.761e-06 & 2.99 & 7.567e-04 & 1.99 \\ \hline 
\end{tabular}
\label{tab:CIB-C}
\end{center}
\end{table}	

\section{Conclusions}\label{con}
This paper tries to answer the question of solving second-order PDEs on domains with curved Lipschitz continuous boundaries or interfaces by high-order methods if the exact solutions are smooth enough. Our method has arbitrary high order, doesn't introduce geometrical errors, and guarantees optimal convergence rates. Also, basis functions are easy to construct, don't depend on the smoothness of sides/faces, and are consistent in 2D and 3D. Though not given here, the weak Galerkin schemes, for Poisson's equation with Neumann boundary condition or the elliptic interface problem with non-homogeneous jump condition (same as in \cite{Lin16}) also have optimal convergence rates and can be proved similarly. There is still room to improve the current method by reducing the unknowns on the sides/faces as in \cite{mu2015weak}, and it provides a way to deal with small edges/faces in the mesh by combining them and treat the connected lines/faces as a whole side/face. Our next steps are to extend the work to PDEs with nonlinear boundary conditions on curved boundaries or interfaces, which are common in biological models, see \cite{guan2022modeling}, and design similar basis functions on curved sides/faces for fourth-order problems.
% good properties
% drawbacks: computational cost
% applications
% future work

\section*{Acknowledgments}
The second author is funded by Grant R01EB034143. The third author has received support from the National Natural Science Foundation of China (Grant Number: 12001325).
%\section*{References}
\bibliography{mybibfile}

\end{document}